\documentclass[reqno,a4paper]{amsart}
\usepackage{amscd,amssymb,amsmath,mathrsfs,latexsym,upgreek,url}
\usepackage[all]{xy}
\usepackage{color,graphicx}
\usepackage{caption,subcaption,enumitem}
\usepackage{hyperref}
\hypersetup{breaklinks=true}

\def\ov#1{{\overline{#1}}}

\def\wh#1{{\widehat{#1}}}
\def\wt#1{{\widetilde{#1}}}
\def\?{\ ???\ \immediate\write16{}%
\immediate\write16{Warning: There was still a question mark . . . }%
\immediate\write16{}}

\newcommand{\conv}{\operatorname{conv}}

\newcommand{\vol}{\operatorname{vol}}

\newcommand{\inter}{\operatorname{int}}

\newcommand{\car}{\operatorname{char}}
\newcommand{\KK}{\operatorname{K}}

\newcommand{\conc}{\operatorname{conc}}

\newcommand{\avol}{\operatorname{\wh{vol}}}

\newcommand{\dd}{\hspace{1pt}\operatorname{d}\hspace{-1pt}}

\renewcommand{\div}{\operatorname{div}}

\newcommand{\stab}{\operatorname{stab}}

\newcommand{\tors}{\operatorname{tors}}

\renewcommand{\i}{\operatorname{i}}

\renewcommand{\Im}{\operatorname{im}}
\newcommand{\e}{{\rm e}}
\newcommand{\h}{\operatorname{h}}

\newcommand{\Hom}{\operatorname{Hom}}

\newcommand{\FS}{{\operatorname{FS}}}
\newcommand{\can}{{\operatorname{can}}}
\newcommand{\ord}{{\operatorname{ord}}}

\newcommand{\val}{{\operatorname{val}}}

\newcommand{\an}{{\operatorname{an}}}

\newcommand{\Aut}{{\operatorname{Aut}}}

\newcommand{\abs}{{\operatorname{abs}}}
\newcommand{\ess}{{\operatorname{ess}}}

\def \C{\mathbb{C}}
\def \F{\mathbb{F}}
\def \G{\mathbb{G}}

\def \K{\mathbb{K}}
\def \L{\mathbb{L}}
\def \N{\mathbb{N}}
\def \P{\mathbb{P}}
\def \Q{\mathbb{Q}}
\def \R{\mathbb{R}}
\def \SS{\mathbb{S}}

\def \T{\mathbb{T}}
\def \Z{\mathbb{Z}}

\def\cF {{\mathcal F}}

\def\cO {{\mathcal O}}

\newcommand{\bfalpha}{{\boldsymbol{\alpha}}}

\numberwithin{equation}{section}
\theoremstyle{definition}
\newtheorem{defn}{Definition}
\numberwithin{defn}{section}

\newtheorem{setting}[defn]{Setting}
\newtheorem{rem}[defn]{Remark}
\newtheorem{exmpl}[defn]{Example}
\newtheorem*{exmpl*}{Example}

\theoremstyle{plain}
\newtheorem{lem}[defn]{Lemma}

\newtheorem{prop}[defn]{Proposition}
\newtheorem{thm}[defn]{Theorem}
\newtheorem{cor}[defn]{Corollary}
\newtheorem{prop-def}[defn]{Proposition-Definition}

\newtheorem{thma}{Theorem}
\newtheorem{cora}[thma]{Corollary}

\begin{document}

\title[Successive minima of toric
  height functions]{Successive minima of toric
  height functions}

\author[Burgos Gil]{Jos\'e Ignacio Burgos Gil}
\address{Instituto de Ciencias Matem\'aticas (CSIC-UAM-UCM-UCM3).
  Calle Nicol\'as Ca\-bre\-ra~15, Campus UAB, Cantoblanco, 28049 Madrid,
  Spain} 
\email{burgos@icmat.es}
\urladdr{\url{http://www.icmat.es/miembros/burgos/}}
\author[Philippon]{Patrice Philippon}
\address{Institut de Math{\'e}matiques de
Jussieu -- U.M.R. 7586 du CNRS, \'Equipe de Th\'eorie des Nombres.
Case 247, 4 place Jussieu, 75252 Paris cedex 05, France}
\email{pph@math.jussieu.fr}
\urladdr{\url{http://www.math.jussieu.fr/~pph}}
\author[Sombra]{Mart{\'\i}n~Sombra}
\address{ICREA \& Universitat de Barcelona, Departament d'{\`A}lgebra i Geometria.
Gran Via 585, 08007 Bar\-ce\-lo\-na, Spain}
\email{sombra@ub.edu}
\urladdr{\url{http://atlas.mat.ub.es/personals/sombra}}

\thanks{Burgos Gil was partially supported by the MICINN research
  projects MTM2009-14163-C02-01 and MTM2010-17389. Philippon was
  partially supported by the CNRS international projects for
  scientific cooperation (PICS) ``Properties of the heights of
  arithmetic varieties'' and ``G\'eom\'etrie diophantienne et calcul
  formel'' and the ANR research project ``Hauteurs, modularit\'e,
  transcendance''.  Sombra was partially supported by the MICINN
  research project MTM2009-14163-C02-01 and the MINECO research
  project MTM2012-38122-C03-02. }

\date{\today} 

\begin{abstract}
  Given a toric metrized $\R$-divisor on a toric variety over a global
  field, we give a formula for the essential minimum of the associated
  height function. Under suitable positivity conditions, we also give
  formulae for all the successive minima. We apply these results to
  the study, in the toric setting, of the relation between the
  successive minima and other arithmetic invariants like the height
  and the arithmetic volume.  We also apply our formulae to compute
  the successive minima for several families of examples, including
  weighted projective spaces,  toric
  bundles and translates of subtori.
\end{abstract}
\maketitle

\overfullrule=3mm
\vspace{-8mm}

\tableofcontents
\vspace{-8mm}

\section[Introduction]{Introduction}

The height is a tool that is ubiquitous in Diophantine geometry and
approximation. It plays a central r\^ole in the proof of finiteness
results on integral and rational points on curves and Abelian
varieties like the theorems of Siegel, Mordell-Weil and Faltings, see
for instance \cite{HindrySilverman:dg,BombieriGubler:heights}. It also
very useful in transcendence theory and in the context of Schmidt's
subspace theorem.

Arakelov geometry provides a convenient framework to define and study
heights. Let $\K$ be a global field, that is, a field which is either a
number field or the function field of a regular projective curve, and let $X$
be an
algebraic variety over $\K$ of dimension $n$. To an (adelically)
metrized $\R$-divisor $\ov D$ on $X$ one can associate a real-valued
height function 
\begin{displaymath}
  \h_{\ov D}\colon X(\ov \K)\longrightarrow \R
\end{displaymath}
on the set of algebraic points of $X$, see \S~\ref{sec:succ-algebr-minima} for details. It is a generalization
of the notion of height of algebraic points considered by 
Northcott, Weil and others.

Given $\eta \in \R$, we denote by $X(\ov \K)_{\le \eta}$ the
set of algebraic
points $p\in X(\ov \K)$ with $\h_{\ov D}(p)\le \eta$. 
For $i=1,\dots, n+1$, the \emph{$i$-th minimum} of $X$ with respect to $\ov
D$ is defined as  
\begin{displaymath}
  \upmu_{\ov D}^{i}(X)= \inf\{\eta \in \R \mid \dim\big(\ov {X(\ov
    \K)_{\le \eta}}\big)
  \ge n-i+1\}.
\end{displaymath}
In particular, the first minimum is the infimum of the real numbers
$\eta$ such that the set $X(\ov \K)_{\le \eta}$ is dense. It is also
called the \emph{essential minimum} of $X$ with respect to $\ov D$,
and denoted $\upmu^{\ess}_{\ov D}(X)$.

These successive minima contain important information on the height
function.  The effective version of the generalized Bogomolov
conjecture asks for an explicit lower bound for the essential minimum
of certain varieties in terms of geometric and arithmetic data. Such
lower bounds have been extensively studied and have several applications
in Diophantine geometry and computer algebra, see for instance
\cite{AvendanoKrickSombra:fbslp,AmorosoViada:spst}.

Our aim in this text is to study the successive minima of height
functions in the toric setting.  Toric objects can be described in
combinatorial terms, and their algebro-geometric properties can be
expressed and studied in terms of this description. In particular, a
proper toric variety~$X$ of dimension $n$ over an arbitrary field is
given by a fan $\Sigma$ on a vector space $N_{\R}\simeq
\R^{n}$. Recall that a toric variety is called proper if it is proper as an
algebraic variety. In combinatorial terms, this is equivalent to the fan being
complete, that is, that the union of its cones covers the whole
vector space. A toric $\R$-divisor~$D$ on a proper toric variety $X$ defines a polytope
$\Delta_{D}$ in the dual space $M_{\R}:=N_{\R}^{\vee}$.  There is a
``toric dictionary'' that translates algebro-geometric properties of
the pair $(X,D)$ into combinatorial properties of the fan and the
polytope. 

In \cite{BurgosPhilipponSombra:agtvmmh,
  BurgosMoriwakiPhilipponSombra:aptv}, we started a program to extend
this toric dictionary to the arithmetic aspects of toric
varieties. Suppose that $X$ is a proper toric variety over the global
field $\K$. Then, to a toric metrized $\R$-divisor $\ov D$ on $X$ we
associate a family of concave functions on the polytope
$\vartheta_{\ov D, v}\colon \Delta_{D}\to \R$, indexed by the places
$\mathfrak{M}_{\K}$ of $\K$. These functions are called the
\emph{local roof functions} of $\ov D$ and they are zero except for a
finite set of places. The \emph{global roof function} $\vartheta_{\ov
  D}$ is the concave function on $\Delta _{D}$ defined as a weighted
sum over all places of these local roof functions. The main theme of
this program is that the global roof function is the arithmetic
analogue of the polytope and encodes a lot of information of the pair
$(X,\ov D)$. Among other results, we gave formulae for the height
$\h_{\ov D}(X)$ and the arithmetic volumes $\wh \vol(\ov D)$ and $\wh
\vol_{\chi}(\ov D)$ in terms of this function.

Our first main result in this text is that the essential minimum of a
toric metrized $\R$-divisor is given by the maximum of the
global roof function.

\begin{thma}[Corollary~\ref{cor:4}]\label{thm:5}
Let $X$ be a proper toric variety over $\K$ and  $\ov D$  a
toric metrized $\R$-divisor on $X$. Then
\begin{displaymath}
  \upmu^{\ess}_{\ov D}(X)=  \max_{x\in
   \Delta _{D}}\vartheta_{\ov D}(x).
\end{displaymath}
\end{thma}

Our second main result is that, under suitable positivity hypothesis
on $\ov D$, not only the essential minimum, but all the succesive
minima can be read from the global roof function. 

\begin{thma}[Theorem \ref{thm:4}]\label{thm:7}
Let $X$ be a proper toric variety over $\K$ and  $\ov D$  a
semipositive toric metrized $\R$-divisor on $X$ with $D$ ample.
Then, for $i=1,\dots,n+1$,  
  \begin{equation*}
  \upmu^{i}_{\ov D}(X) = \min_{F\in \cF(\Delta_{D})^{n-i+1}} \max_{x\in
    F}\vartheta_{\ov D}(x),
  \end{equation*}
  where $\cF(\Delta_{D})^{n-i+1}$ is the set of faces of the polytope
  $ \Delta_{D}$ of dimension $n-i+1$.
\end{thma}

Whereas there is a considerable amount of work on upper and lower
bounds for the essential minimum, there are very few exact
computations in the literature. By contrast, Theorems \ref{thm:5} and
\ref{thm:7} are very concrete and well-suited for computations.  For
example, they allow to compute the successive minima of the canonical
height on translates of subtori of a projective space as the maximum
of a piecewise affine concave function on the polytope (Proposition
\ref{prop:10}). The following example illustrates this computation.

\begin{exmpl*}
Let  $C \subset \P^3_{\Q}$ be the cubic curve given as the image
of the  map 
\begin{displaymath}
  \P^1 \longrightarrow \P^3, \quad (t_{0}:t_{1}) \longmapsto
  \Big(t_{0}^{3} : 4\, t_{0}^{2}t_{1}: \frac{1}{3}\, t_{0}t_{1}^2:
  \frac{1}{2} \, t_{1}^3\Big).
\end{displaymath} 
Let $\ov H$ be the metrized divisor of $\P^{3}$ given by the
hyperplane at infinity equipped with the canonical metric, and let $\ov D$ be
the restriction of $\ov H$ to $C$.

Figure \ref{fig:localrf} shows the local roof functions associated to
$\ov D$ for each place $v\in \mathfrak{M}_{\Q}$, and Figure
\ref{fig:globalrf} shows the global roof function. This global roof
function is the sum of the local ones, and can be described as the
minimal concave piecewise affine function on the interval $[0,3]$ with
lattice point values
\begin{displaymath}
    \vartheta_{\ov D}(0)=0, \quad \vartheta_{\ov D}(1)= \frac{7}{3} \log(2)+
  \frac{1}{2}\log(3), \quad \vartheta_{\ov D}(2)= \frac{7}{6} \log(2)+
  \log(3), \quad\vartheta_{\ov D}(3)=0.
\end{displaymath}
\captionsetup[subfigure]{labelformat=empty}
\begin{figure}[ht]
  \centering
  \begin{subfigure}[1]{0.33\textwidth}
    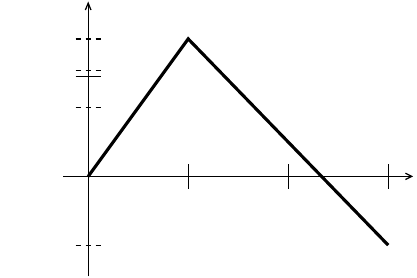
    \caption{$v=\infty$}
    \vspace{5mm}
  \end{subfigure}
  \hspace*{20mm}
  \begin{subfigure}[2]{0.33\textwidth}
    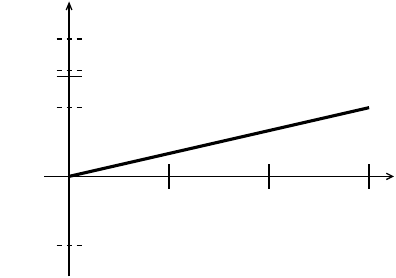    
    \caption{$v=2$}
    \vspace{5mm}
  \end{subfigure}
  \begin{subfigure}[1]{0.33\textwidth}
    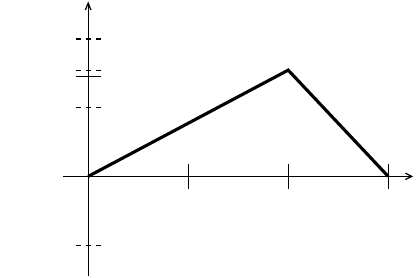
    \caption{$v=3$}
  \end{subfigure}
  \hspace*{20mm}
  \begin{subfigure}[2]{0.33\textwidth}
    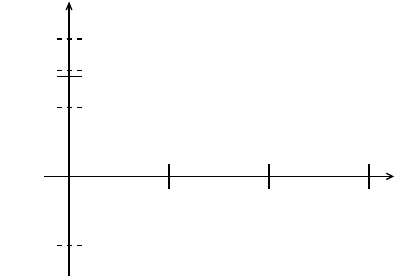    
    \caption{$v\ne\infty, 2,3$}
  \end{subfigure}
\caption{Local roof functions}
\label{fig:localrf}
\end{figure}

\begin{figure}[ht] 
  \centering
  \begin{subfigure}[1]{0.3\textwidth}
    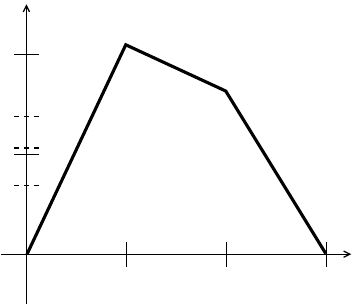
  \end{subfigure}
\caption{Global roof function}
\label{fig:globalrf}
\end{figure}
Theorem \ref{thm:7} then implies  that
\begin{math}
\upmu_{\ov D}^{\ess}(C)= \frac{7}{3} \log(2)+
  \frac{1}{2}\log(3)\text{ and }\upmu_{\ov D}^{2}(C)=0. 
\end{math}
We refer to \S~\ref{sec:transl-subt-can} for explanations on how to do
this kind of computations.
\end{exmpl*}

Our results allow also to compute the successive minima of toric
varieties with respect to weighted $L^{p}$-metrics and of
translates of subtori of a projective space with the Fubini-Study
metric, generalizing the computation of the successive minima of a
subtori with the Fubini-Study metric in \cite{Sombra:msvtp}. Another
nice family of examples is given by toric bundles on a projective
space, including Hirzebruch surfaces. We refer the reader to
\S~\ref{sec:exmpl} for the details and the explicit formulae.

A well-known theorem of Zhang shows that the successive minima of a
height function can be estimated in terms of the height and the degree
of the variety \cite{Zhang:plbas,Zhang:_small}, This result plays a
key r\^ole in the proof of the
Bogomolov conjecture for Abelian varieties and its ulterior
developments, including the
study of the distribution of Galois orbits of points of small height,
see for instance \cite{DavidPhilippon:mhn, Yuan:blbav}.

As a direct consequence of Theorems \ref{thm:5} and \ref{thm:7} and
our previous results in \cite{BurgosPhilipponSombra:agtvmmh,
  BurgosMoriwakiPhilipponSombra:aptv}, we obtain a simple proof of
Zhang's theorem in the toric case. This approach allows also to prove
this result for an arbitrary global field and to relax the positivity
hypothesis on the metrized $\R$-divisor. 

\begin{thma}[Theorem \ref{thm:6}] \label{thm:8} Let $X$ be a proper
  toric variety over $\K$
  of dimension $n$ and $\ov D$ a semipositive toric metrized
  $\R$-divisor on $X$ with $D$ big. Then 
  \begin{equation*}
\sum_{i=1}^{n+1}\upmu^{i}_{\ov D}(X)\le \frac{\h_{\ov D}(X)}{ \deg_{D}(X)}\le (n+1)
\upmu^{\ess}_{\ov D}(X).
  \end{equation*}
\end{thma}

Using our formulae, we can easily construct examples of semipositive
toric metrics on the hyperplane divisor on $\P^{n}_{\Q}$ showing that
almost every configuration of successive minima and height can
actually happen. 

\begin{thma}[Proposition \ref{prop:11}] \label{prop:12}
  Let $n\ge 0$ and $\nu, \mu_{1},\dots, \mu_{n+1}\in \R$ such that 
  \begin{equation}\label{eq:40}
    \mu_{1}\ge \dots\ge \mu_{n+1} \quad \text{ and } \quad
    \sum_{i=1}^{n+1}\mu_{i}\le \nu <(n+1) \mu_{1}. 
  \end{equation}
  Then there exists a semipositive toric metric on $H$, the hyperplane
  divisor on $\P^{n}_{\Q}$, such that
\begin{equation*}
  \h_{\ov H}(\P^{n})=\nu \quad \text{ and } \quad   \upmu_{\ov
    H}^{i}(\P^{n})=\mu_{i} \quad \text{ for } i=1,\dots, n+1.
\end{equation*}
\end{thma}

The case left aside in \eqref{eq:40} deserves a separate explanation:
we show that if, with the hypothesis of Theorem \ref{thm:8}, we have
the equality
\begin{displaymath}
  \frac{\h_{\ov D}(X)}{ \deg_{D}(X)}= (n+1)
\upmu^{\ess}_{\ov D}(X),
\end{displaymath}
then the function $\vartheta _{\ov D}$ is constant and, necessarily, 
$\upmu_{\ov D}^{i}(X)=\upmu_{\ov D}^{\ess}(X)$ for all~$i$, see
Corollary \ref{cor:5} below. This observation is relevant when
applying the known equidistribution results on Galois orbits of small
points in the toric case.

By replacing the height of the variety by its $\chi$-arithmetic
volume, Zhang's lower bound for the essential minimum extends to the
case when the metrics are not necessarily semipositive: if $X$ is a
proper variety of dimension $n$ over a number field and $\ov D$ is a
metrized divisor on $X$ with $D$ big, then
\begin{equation} \label{eq:39}
\upmu^{\ess}_{\ov D}(X)\ge     \frac{\avol_{\chi}(\ov D)}{(n+1) \vol(D)},
  \end{equation}
  see for instance \cite[Lemme 5.1]{ChambertLoirThuillier:MMel}.  This
  lower bound is a key result in the study of the distribution of the
  Galois orbits of points of small height. Indeed, all known results
  in this direction are applicable only when the inequality
  \eqref{eq:39} is an equality. This includes the equidistribution theorems of
  Szpiro, Ullmo and Zhang \cite{SzpiroUllmoZhang:epp}, Bilu
  \cite{Bilu:ldspat}, Favre and Rivera-Letelier
  \cite{FavreRivera:eqpph}, Chambert-Loir \cite{ChambertLoir:meeB},
  Baker and Rumely \cite{BakerRumely:esp}, Yuan \cite{Yuan:blbav},
  Berman and Boucksom \cite{BermanBoucksom:gbhsee}, and Chen
  \cite{Chen:davf}.

In the toric case, we can also derive a simple proof of
  the inequality  \eqref{eq:39} for arbitrary
  global fields and toric metrics on a big toric
  $\R$-divisor. More importantly, we can characterize the cases when
  equality occurs. The following statement is a direct consequence of
  Propositions 
  \ref{prop:5} and \ref{prop:6}.

\begin{cora} \label{cor:5} Let $X$ be a proper toric variety over $\K$
  of dimension $n$ and $\ov D$ a toric metrized $\R$-divisor on $X$
  with $D$ big. Then
\begin{equation*}
\upmu^{\ess}_{\ov D}(X)\ge     \frac{\avol_{\chi}(\ov D)}{(n+1) \vol(D)},
  \end{equation*}
with equality if and only if $\vartheta_{\ov D}$ is constant. 
If this is the case, then $\upmu_{\ov D}^{i}(X)=\upmu_{\ov D}^{\ess}(X)$ for all~
$i$. 
\end{cora}

The condition that the roof function is constant is \emph{very}
strong, and it is equivalent to the fact that the $v$-adic metrics in
$\ov D$ can be derived from the canonical metric by translation and
scaling (Remark \ref{rem:2}). In particular, these are the only
toric metrics to which the equidistribution theorems mentioned above
can be applied.

In collaboration with Juan Rivera-Letelier, we are currently studying
the equidistribution properties of Galois orbits of points of small
height in the toric setting. We plan to expose our results in a
subsequent paper. 

\medskip \noindent {\bf Acknowledgements.} We thank Dominique Bernardi,
Pierre D\`ebes, Luis Dieulefait and Juan Carlos Naranjo for useful
discussions and pointers to the literature. 

Part of this work was done while the authors met at the Universitat de
Barcelona, the Instituto de Ciencias Matem\'aticas (Madrid) and the
Institut de Math\'ematiques de Jussieu (Paris). 

\section{Preliminary results}
\label{sec:preliminary-results}

In this section we gather some preliminary results on global fields
and on convex analysis.

\subsection{Global fields}
\label{sec:global-fields}
A \emph{global field} $\K$ is either a number field or the function
field of a regular projective curve over an arbitrary field, equipped
with a set of places $\mathfrak{M}_{\K}$.  Each place $v\in
\mathfrak{M}_{\K}$ is a pair consisting of an absolute value
$|\cdot|_{v}$ on $\K$ and a positive weight $n_{v}\in \Q_{> 0}$,
defined as follows.

The places of the field of rational numbers $\Q$ consist of the
Archimedean and the $p$-adic absolutes values, normalized in the
standard way, and with all weights equal to 1. For the function field
$\KK(C)$ of a regular projective curve $C$ over a  field $k$,
the set of places is indexed by the closed points of $C$. For each
closed point $v_{0}\in C$, we consider the
absolute value and weight given, for $\alpha\in \KK(C)^{\times}$, by
\begin{displaymath}
  |\alpha|_{v_{0}}=c_{k}^{-\ord_{v}(\alpha)}, \quad n_{v_{0}}=[k(v_{0}):k], 
\end{displaymath}
where $\ord_{v_{0}}(\alpha)$ denotes the order of $\alpha$ in the discrete
valuation ring $\cO_{C,v_{0}}$ and 
\begin{displaymath}
  c_{k}=
  \begin{cases}
    \e & \text{ if } \# k=\infty, \\
\# k& \text{ if } \# k<\infty.
  \end{cases}
\end{displaymath}

Let $\K_{0}$ denote either $\Q$ or $\KK(C)$, and let $\K$ be a finite
extension of $\K_{0}$. The set of places of $\K$ is then formed by the
pairs $v=(|\cdot|_{v},n_{v})$ where $|\cdot|_{v}$ is an absolute value
on $\K$ extending an absolute value $|\cdot|_{v_{0}}$ on $\K_{0}$ and
\begin{equation}\label{eq:26}
  n_{v}=\frac{[\K_{v}:\K_{0,v_{0}}]}{[\K:\K_{0}]} n_{v_{0}},
\end{equation}
where $\K_{v}$ denotes the completion of $\K$ with respect to
$|\cdot|_{v}$, and similarly for $\K_{0,v_{0}}$. By \cite[Proposition
XII.6.1]{Lang:Algebra}, the weight $n_{v}$ can be also written as
\begin{equation}\label{eq:48}
  n_{v}=\frac{e({v/v_{0}})f({v/v_{0}})}{[\K:\K_{0}]} n_{v_{0}},
\end{equation}
where   $e({v/v_{0}})$ is the ramification degree and $f({v/v_{0}})$ is
the residue class degree of $v$ over $v_{0}$. Therefore, the notion of
global field in this paper is compatible with that in \cite[Definition 
  1.5.12]{BurgosPhilipponSombra:agtvmmh}.

  In the function field case, the extension $\K/\K_{0}$ corresponds to
  a dominant morphism $B\to C$ of regular projective curves and
  $\K=\KK(B)$.  However, the set of places of $\K$ depends on the
  extension and not just on the field $\K$.

Given  $v\in \mathfrak{M}_{\K}$ and $v_{0}\in \mathfrak{M}_{\K_{0}}$,
we write $v\mid v_{0}$ whenever $|\cdot|_{v}$ extends
$|\cdot|_{v_{0}}$. 
The  set of
places of $\K$ satisfies, for all $v_{0}\in \mathfrak{M}_{\K_{0}}$,
\begin{equation}
  \label{eq:46}
  \sum_{v\mid v_{0}}n_{v}=n_{v_{0}}
\end{equation}
and the \emph{product formula}
\begin{displaymath}
  \prod_{v\in \mathfrak{M}_{\K}}|\alpha
  |_{v}^{n_v}=1\quad\text{ for all }\alpha \in \K^{\times}.
\end{displaymath}
Both properties are well-known in the case of number fields. In the
function field case, the equality  \eqref{eq:46} follows from the
projection formula \cite[Proposition~9.2.11]{Liu:agac}, whereas the
product formula for $\K$ follows from \eqref{eq:46} and the product
formula for $\K_{0}$. 

We will first construct finite extensions of $\K$ of arbitrary degree that
are totally split over a given set of places. To this end, we need the
following consequence of Hilbert's irreducibility theorem.

\begin{lem}\label{hilbirred}
  Let $f(x)\in\K[x]$ be a separable polynomial of positive degree,
  $S\subset{\mathfrak M}_\K$ a finite subset of places of $\K$ and
  $(\varepsilon_v)_{v\in S}$ a collection of positive real
  numbers. Then there exists an element $c\in\K$ such that the
  polynomial $f(x)+c$ is irreducible in $\K[x]$ and $|c|_v<
  \varepsilon_v$ for all $v\in S$.
\end{lem}
\begin{proof} We want to use Hilbert's irreducibility theorem for
  fields with a product formula in~\cite[Theorem
  3.4]{Debes:DRHS}. To this end, we need to show that any global field
  satisfies its hypothesis. The first hypothesis is that the field is
  either of characteristic zero or imperfect of positive
  characteristic. Since, if $\car(\K)>0$, then $\K$ is the function
  field of a curve over a field of positive characteristic and so it
  is not perfect. Thus, this hypothesis holds for global fields.  The
  second one is a density hypothesis that, when $\K$ is a number
  field, follows from the strong approximation theorem and, when $\K$
  is a function field, follows from the Riemann-Roch theorem for curves
  over an arbitrary field given in~\cite[Theorem 7.3.17]{Liu:agac}.

Consider the polynomial
$$F(x,y) = f(x)+y\in \K[x,y].$$
Being irreducible in $\K[x,y]$ and of positive degree in $x$, it is
also irreducible in $\K(y)[x]$.  
Then \cite[Theorem 3.4]{Debes:DRHS} implies that there exists $c\in\K$ such that
$F(x,c)=f(x)+c$ is irreducible in $\K[x]$ and $|c|_v \leq
\varepsilon_v$ for all $v\in S$ as stated.
\end{proof}

\begin{lem}\label{constrextotdec}
Let $d\geq 1$ be an integer and
$S\subset\mathfrak{M}_\K$ a finite subset of places of $\K$. There exists an
extension $\L/\K$ of degree $d$ such that, for all $v\in S$, there are
$d$ different extensions of the absolute value $|\cdot|_v$ to
$\L$.
\end{lem}
\begin{proof}
Let $\beta_1,\dots,\beta_d \in\K$ such that $\beta _{i}\not=\beta_{j}$
for $i\not=j$. Set $f(x)=\prod_{j=1}^d(x-\beta_i) \in
\K[x]$, which  is a separable polynomial of positive degree. For $v\in
S$, put
\begin{displaymath}
  \varepsilon _{v}=\Big(\frac{1}{4}\min_{i\not = j}|\beta _{i}-\beta
    _{j}|_{v}\Big)^{d}. 
\end{displaymath}
By Lemma~\ref{hilbirred}, there is an element $c\in \K$ such that
$f(x)+c$ is irreducible and $|c|_{v}<\varepsilon _{v}$ for $v\in
S$. Set 
\begin{displaymath}
\L=\K[x]/(f(x)+c).  
\end{displaymath}

Since $f(x)$ is monic of degree $d$, so is $f(x)+c$.  Denote by
$\alpha_1,\dots,\alpha_d$ the roots of $f(x)+c$ in an algebraic
closure of $\K$. For $v\in S$ and $i\in \{1,\dots, d\}$ we have that $|f(\beta
_i)+c|_v = |c|_v$, from which it follows that there exists an index
$\sigma(v,i)\in\{1,\dots,d\}$ satisfying
\begin{equation}\label{eq:majrac}
|\alpha_{\sigma(v,i)}-\beta_{i}|_v \leq |c|_v^{1/d} < \varepsilon_v^{1/d}.
\end{equation}
By the choice of $\varepsilon _{v}$, we deduce that $\sigma
(v,i)\not=\sigma (v,j)$ for $i\not=j$, and so  $\sigma (v,\cdot)$ is
a bijection. Let $\tau (v,\cdot)$ denote the inverse bijection. Then,
using \eqref{eq:majrac} and the definition of $\varepsilon_v$, we
obtain, for $i\in \{1,\dots,d\}$ and $j\not= i$,
\begin{multline}\label{eq:9}
  |\alpha _{i}-\alpha _{j}|_{v} > |\beta _{\tau
    (v,i)}-\beta _{\tau (v,j)}|_{v}-2\varepsilon_v^{1/d}
    \ge 2\varepsilon_{v}^{1/d}
    > 2|\alpha_{i}-\beta_{\tau (v,i)}|_{v}.  
\end{multline}
This implies
that $f(x)+c$ is separable. Moreover,  the inequality \eqref{eq:9} also implies
that, for each $i\in\{1,\dots,d\}$, we have $\K_{v}(\alpha
_{i})=\K_v(\beta_{\sigma(v,i)})=\K_v$. If $v$ is non-Archimedean, this
follows from Krasner's lemma \cite[page 152]{Neukirch:ANT}. If
$v$ is Archimedean, we only need to see
that, if $\K_v=\R$, then $\K_{v}(\alpha _{i})=\R$. Assume that, on the
contrary, 
$\K_{v}(\alpha _{i})=\C$. Since the coefficients of $f(x)+c$ are real,
there is  $j\not=i$ such that $\alpha _{j}$ is the complex conjugate
of $\alpha _{i}$. By hypothesis, $\beta_{\sigma(v,i)}\in \K_{v}=\R$
and so 
\begin{displaymath}
  |\alpha _{j}-\alpha _{i}|_{v}\le 2|\alpha _{i}-\beta_{\tau (v,i)}|_v<
  \min_{\stackrel{1\leq j\leq
    d}{j\not=i}}|\alpha_i-\alpha_j|_v,
\end{displaymath}
which is a contradiction.
 
Thus, for all $v\in S$, the polynomial $f(x)+c$ splits completely in
$\K_v[x]$. Then, by~\cite[Proposition 8.2]{Neukirch:ANT}, this implies
that there are $d$ distinct places of the extension
$\L=\K[x]/(f(x)+c)$ over~$v$,  completing the proof.
\end{proof}

Let $\G_{m}$ be the multiplicative group over $\K$ and $\T\simeq
\G^{n}_{m}$ a split torus of dimension $n$ over $\K$. Let
$N=\Hom(\G_{m},\T)$ be the lattice of cocharacters of $\T$ and write
$N_{\R}=N\otimes \R$.  We fix a splitting $\T\simeq \G_{m}^{n}$, which
induces isomorphisms $\T(\K)\simeq (\K^{\times})^{n}$ and
$N_{\R}\simeq \R^{n}$. Given elements $x\in \T(\K)$ and $u\in N_{\R}$,
we denote by $x_{i}$ and $u_{i}$, $i=1,\dots,n$, the components of the
image of $x$ and $u$ under the previous isomorphisms.  Consider the
space $\bigoplus _{v\in \mathfrak{M}_{\K}}N_{\R}$ with the norm given
by
\begin{displaymath}
  \|(u_{v})_{v}\|=\sum_{v\in \mathfrak{M}_{\K}}n_{v}
  \sum_{i=1}^{n}|u_{v,i}|.
\end{displaymath}
The induced topology is called the \emph{$L^{1}$-topology}. It does
not depend on the choice of the splitting of $\T$. We denote by
$H_{\K}\subset \bigoplus _{v\in \mathfrak{M}_{\K}}N_{\R}$ the subspace
defined by
  \begin{equation}\label{eq:16}
    H_{\K}=\Big\{(u_{v})_{v}\in\bigoplus _{v\in
      \mathfrak{M}_{\K}}N_{\R} \, \Big|\  \sum_{v}n_{v}u_{v}=0\Big\}
  \end{equation}
  with the induced $L^{1}$-topology. 

For each $v\in
  \mathfrak{M}_{\K}$, there is a map $\val_{v}\colon \T(\K)\to N_{\R}$,
  given, in the fixed  splitting, by
  \begin{equation}\label{eq:44}
    \val_{v}(x_{1},\dots,x_{n})=(-\log|x_{1}|_{v},\dots,-\log|x_{n}|_{v}).
  \end{equation}
  This map does not depend on the choice of the splitting.
  By the product formula, we can define a map $\val\colon \T(\K)\to
  H_{\K}$ as
  \begin{displaymath}
    \val(x)=(\val_{v}(x))_{v\in \mathfrak{M}_{\K}}.
  \end{displaymath}
This is a group homomorphism, and so it can be extended to a map
\begin{displaymath}
  \val\colon \T(\K)\otimes \Q\to H_{\K}.
\end{displaymath}

Dirichlet's unit theorem does not hold for general global fields. 
Nevertheless, the following result, that in the case of number fields is
a consequence of Dirichlet's unit theorem, is true in general.

\begin{lem}\label{lemm:7} The set $\val(\T(\K)\otimes\Q)$ is dense in
  $H_{\K}$ with respect to the $L^{1}$-topology.
\end{lem}

\begin{proof}
  Since the torus is split, working component-wise, it is enough to
  treat the case $n=1$. Thus $\T(\K)=\K^{\times}$.

  First suppose that $\K$ is an number field or the function field of
  a curve over a finite field. For each finite subset
  $S\subset \mathfrak{M}_{\K}$ we put
  \begin{align*}
    H_{\K,S}&=\{(u_{v})_{v\in\mathfrak{M}_{\K}}\in H_{\K}\mid u_{v}=0 \text{ for }v\not
    \in S\}, \\
  \K_{S}& =\{\alpha \in \K\mid |\alpha |_{v}=1 \text{ for }v\not
    \in S\}\subset \K^{\times}.
 \end{align*}
 Dirichlet $S$-unit theorem \cite[Chapter IV, \S~4, Theorem
 9]{Weil:bnt} states that $\val(\K_{S})$ is a lattice in
 $H_{\K,S}$. Let $u\in H_{\K}$. Then there exists a finite subset $S$
 such that $u\in H_{\K,S}$. Let $\varepsilon >0$. By the density of
 rational numbers, we can find an element $u'\in \val(\K_{S}\otimes
 \Q)\subset H_{\K,S}$ with $\|u-u'\|<\varepsilon $, proving the lemma
 in this case.

 Now let $B\to C$ be a dominant morphism of regular projective curves
 over an infinite field $k$ and set $\K=\KK(B)$ with the induced structure of
 global field.  In this case, Dirichlet $S$-unit theorem may not hold
 and the lemma is a consequence of the Riemann-Roch theorem.

  Let $(u_{v})_{v}\in H_{\K}$ and
  $\varepsilon >0$. We have to  
  show that there is an element $x\in \K^{\times}\otimes \Q$ such that
  \begin{displaymath}
    \|(u_{v})_{v}-\val(x)\|<\varepsilon.
  \end{displaymath}
  Since $\Q$ is dense in $\R$, we may assume without loss of generality
  that $u_{v}\in \Q$ for all $v\in \mathfrak{M}_{\K}$. Since there is
  a finite subset $S$ such that $u\in H_{\K,S}$, we can choose an integer
  $q\ge 1$ such
  that $qu_{v}\in \Z$ for all $v\in 
  \mathfrak{M}_{\K}$. 

  Recall that the set of places of $\K$ is indexed by the set of
  closed points of $B$. With notation as in \eqref{eq:48}, we consider
  the Weil divisor on $B$ given by
  \begin{displaymath}
    D=\sum_{v\in \mathfrak{M}_{\K}}e(v/v_{0}) qu_{v}[v].
  \end{displaymath}
By the definition of the weights $n_{v}$ and the product formula,
 \begin{displaymath}
   \deg(D)=\sum_{v\in
     \mathfrak{M}_{\K}} e(v/v_{0}) qu_{v} [k(v):k] = 
   q[\K:\K_{0}] \sum_{v\in \mathfrak{M}_{\K}}n_{v}u_{v}=0.
 \end{displaymath}
Let $E$ be an effective Weil divisor on $B$ with $\deg(E)=r\ge g(B)$,
 where $g(B)$ is the genus of $B$, and choose an integer 
$$
l> \frac{2r}{\varepsilon q[\K:\K_{0}]}.
$$
 Since $\deg(lD+E)=l\deg(D)+\deg(E)= r\ge g(B)$, by the
 Riemann-Roch theorem~\cite[Theorem 7.3.17]{Liu:agac} we can find an
 element $\alpha \in \K^{\times}$ and an effective divisor $E'$ on $B$
 with
 $\deg(E')=r$ such that
 \begin{equation}\label{eq:12}
   lD+E=E'+\div(\alpha ).
 \end{equation}
 Writing $E'-E=\sum_{v}a_{v}[v]$, the equation \eqref{eq:12} reads
 \begin{equation}
   \label{eq:55}
 a_{v}= e(v/v_{0}) l qu_{v}-\ord_{v}(\alpha)=e(v/v_{0}) l q\Big(u_{v}-\frac{1}{lq}\val_{v}(\alpha)\Big) \quad \text{ for all } v.
 \end{equation}
Put
 \begin{math}
   x=\alpha^{\frac{1}{lq}}=\alpha\otimes \frac{1}{lq} \in \K^{\times }\otimes \Q.
 \end{math}
Using that $E$ and $E'$ are effective divisors of degree
 $r$, we deduce that
 \begin{displaymath}
 \sum_{v\in
   \mathfrak{M}_{\K}}[k(v):k] |a_v| \le \deg(E)+\deg(E') = 2r  
 \end{displaymath}
 and, using \eqref{eq:55}, 
 \begin{multline*}
   \|(u_{v})_{v}-\val(x)\|=  \Big\|(u_{v})_{v}-\frac{1}{lq}
   \val(\alpha)\Big\|= \frac{1}{lq} \Big\|\Big(\frac{a_{v}}{e(v/v_{0})
   }\Big)_{v}\Big\| \\= 
 \frac{1}{lq [\K:\K_{0}]}\sum_{v\in \mathfrak{M}_{\K}}[k(v):k]|a_v|
   \le \frac{2r}{lq [\K:\K_{0}]}<\varepsilon,
 \end{multline*}
 obtaining the result.
\end{proof}

\begin{rem}
The space $H_{\K}$ has another natural
  topology, the \emph{direct sum topology}.
  A subset $U\subset H_{\K}$ is open for the direct sum topology if
  and only if its intersections with all the subsets $H_{\K,S}$ are
  open. The direct sum topology is finer than the $L^{1}$-topology. In
  fact, a sequence $(u_{j})_{j\ge 1}$ of elements of $H_{\K}$ converges to
  $u\in H_{\K}$ in the direct sum topology if and only it converges in the
  $L^{1}$-topology and there is a finite subset $S\subset
  \mathfrak{M}_{\K}$ such that $u_{j}\in H_{\K,S}$ for all $j\ge 1$.   

  The proof of Lemma \ref{lemm:7} for number fields and function
  fields over a finite field shows the stronger result that the set
  $\val(\T(\K)\otimes\Q)$ is dense in $H_{\K}$ for the direct sum
  topology.  By contrast, the proof of Lemma \ref{lemm:7} for general
  function fields only shows density for the $L^{1}$-topology because
  we have no control on the support of the divisor $E'$ in the
  equation \eqref{eq:12}.
\end{rem}

Although the $L^{1}$-topology is coarser than the direct sum topology,
the next result shows that it will be enough for our purposes.

\begin{lem}\label{lemm:8} Let $\Psi \colon N_{\R}\to \R$ be a
  continuous function with $\Psi (0)=0$ and Lipschitz at $0$. Let $(\psi
  _{v})_{v\in \mathfrak{M}_{\K}}$ be a collection of continuous
  functions on $N_{\R}$ such that  there is a finite subset
  $S\subset \mathfrak{M}_{\K}$ with  $\psi _{v}=\Psi $ for $v\not \in S$. Then   
  the map $\bigoplus _{v\in \mathfrak{M}_{\K}}N_{\R}\to \R$ given by
  \begin{displaymath}
    (u_{v})_{v}\longmapsto \sum_{v\in \mathfrak{M}_{\K}} n_{v}\psi _{v}(u_{v}) 
  \end{displaymath}
  is continuous with respect to the $L^{1}$-topology.
\end{lem}
\begin{proof} First note that the function in the lemma is
  well-defined because the sum only involves a finite number of
  nonzero terms. Within this proof, we will indistinctly denote by  $\| \cdot \|$  the
  $L^{1}$-norm on $N_{\R}\simeq\R^{r}$ or on $\bigoplus_{v}N_{ \R}$. 
  
Since $\Psi $ is Lipschitz at $0$, there are constants
  $B>0$ and $\varepsilon _{0}>0$ such that, for $u'\in 
  N_{\R}$ with $\|u'\|\le \varepsilon _{0}$, 
  \begin{displaymath}
    |\Psi(u')|\le B\|u'\|,
  \end{displaymath}

  Fix $(u_{v})_{v}\in \bigoplus _{v}N_{\R}$ and $\varepsilon
  >0$. Write
  \begin{displaymath}
    S'=S\cup\{v\in \mathfrak{M}_{\K}\mid u_{v}\not = 0\}.
  \end{displaymath}
Since $\psi _{v}$ is continuous in $u_{v}$, we can
  choose $0<\delta 
  <\min(\varepsilon /2B,\varepsilon _{0})$ such that, for all $v\in S'$, 
  \begin{displaymath}
    n_{v}\|u_{v}-u'_{v}\|< \delta \Longrightarrow
    n_{v}|\psi _{v}(u_{v})-\psi_{v}(u'_{v})|<
    \frac{\varepsilon}{2\#S'}.
  \end{displaymath}
  If $\|(u'_{v})_{v}-(u_{v})_{v}\|<\delta $ then
  $n_{v}\|u_{v}-u'_{v}\|< \delta$ for all $v\in
  \mathfrak{M}_{\K}$, and $\sum_{v\notin
    S'}n_v\|u'_v\|<\delta$. Therefore
  \begin{multline*}
    \left|
      \sum_{v\in \mathfrak{M}_{\K}}n_{v}\psi _{v}(u'_{v})
      -\sum_{v\in \mathfrak{M}_{\K}}n_{v}\psi _{v}(u_{v})
      \right|\\
      \le\sum_{v\in S'}n_{v}|\psi _{v}(u'_{v})-\psi
      _{v}(u_{v})|+
      \sum_{v\not \in S'}n_{v}|\Psi (u'_{v})|\\
      <
      \sum_{v\in S'}\frac{\varepsilon}{2\#S'}+
     B\sum_{v\not \in S'}n_{v}\|u'_{v}\|<
      \frac{\varepsilon}{2}+B\delta < \varepsilon. 
  \end{multline*}
This shows the continuity at the point $(u_{v})_{v}$. Since this point
is arbitrary we obtain the lemma.
\end{proof}

\subsection{Concavification of functions}
\label{sec:concavification} 
We next introduce the concavification of a function and study its
basic properties. 

\begin{defn}\label{def:3}
Let $f\colon N_{\R}\to \R$ be a
  function. The \emph{concavification} of $f$, denoted $\conc(f)$, is
  the smallest
  concave function on $N_{\R}$ that is bounded below by 
  $f$.   
\end{defn}

The concavification of a function may not exists but if it exists, it
is unique.  Recall from from \cite[Definition
A.1]{BurgosMoriwakiPhilipponSombra:aptv} that the \emph{stabilizer} of
the function $f$ is the subset of $M_{\R}$ given by
\begin{equation}\label{eq:50}
  \stab(f)= \{x\in M_{\R}\mid x - f \text{ is bounded below}\}. 
\end{equation}

\begin{lem} \label{lemm:2} Let $f\colon N_{\R}\to \R$ be a
  function. Then $\conc(f)$ exists if and only if
  $\stab(f)\ne\emptyset$. If this is the case, then for $u\in
  N_{\R}$,
  \begin{displaymath}
    \conc(f)(u)= \sup \sum_{j=1}^{\ell}\nu_{j}f(u_{j}),
  \end{displaymath}
  where the supremum is over all expressions of $u$ as a convex
  combination of points of $N_{\R}$, that is, all expressions of the
  form $u=\sum_{j=1}^{\ell}\nu_{j} u_{j}$ with $\ell\in\N$,
  $\nu_{j}\ge 0$ for all~$j$,
  $\sum_{j=1}^{\ell}\nu _{j}=1$ and
  $u_{j}\in N_{\R}$.
 \end{lem}

 \begin{proof} Clearly,  $\conc(f)$ exists if and only if
   there exists a concave function $g\colon N_{\R}\to \R$ with $f\le
   g$.

Assume that $\stab(f)\not =\emptyset$. 
  Let $x\in\stab(f)$. Then, there exists $c\in \R$ such that $f(u)\le
  \left<x,u\right>+c$ for all $u\in N_{\R}$. Since the function
  $\left<x,u\right>+c$ is concave, we deduce that $\conc(f)$
  exists. Conversely, assume that $\conc(f)$ exists. Since $\conc(f)$
  is concave, $\stab(\conc(f))\not =\emptyset $. Therefore
  $\stab(f)\supset \stab(\conc(f))$ is not empty.

  The expression for
  $\conc(f)(u)$ follows from \cite[Theorem 5.3]{Rockafellar:ca}, see
  \emph{loc. cit.} page 36.
\end{proof}

If the function $f$ is locally bounded below, we can
assume that the numbers $\nu _{j}$ of the previous lemma are rational
numbers.

\begin{lem}\label{lemm:6}
   Let $f\colon N_{\R}\to \R$ be a function
such that 
  $\stab(f)\ne\emptyset$ and which is locally bounded below. Then, for $u\in
  N_{\R}$,
  \begin{displaymath}
    \conc(f)(u)= \sup \frac{1}{d}\sum_{j=1}^{d}f(u_{j}),
  \end{displaymath}
  where the supremum is over all expressions of the
  form $u=\frac{1}{d}\sum_{j=1}^{d}u_{j}$ with 
  $u_{j}\in N_{\R}$.
\end{lem}

\begin{proof}
  By Lemma \ref{lemm:2}, it is clear that 
  \begin{displaymath}
    \conc(f)(u)\ge \sup \frac{1}{d}\sum_{j=1}^{d}f(u_{j}).
  \end{displaymath}
  Thus, we only need to show the other inequality. Let $\varepsilon
  >0$. By Lemma \ref{lemm:2} we can find a convex combination
  $u=\sum_{j=1}^{k}\nu_{j}u_{j}$ with $\nu_{j}> 0$ for all~$j$ and
  $\sum_{j=1}^{\ell}\nu _{j}=1$ such that
 \begin{equation}\label{eq:57}
    \conc(f)(u)\le \sum_{j=1}^{k}\nu _{j}f(u_{j})+\varepsilon /2.
  \end{equation}
 
Fix an isomorphism $N_{\R}\simeq \R^{n}$ and consider the associated
$L^{1}$-norm, that we denote by $\|\cdot\|$. 
Since $\stab(f)\not=\emptyset$ and $f$ is locally bounded below, the
  function $|f|$ is bounded on compact subsets. Therefore, 
  there is a constant $B>0$ such that $|f(v)|\le B$ for all $v\in N_{\R}$ with
  $\|v\|\le 2\sum_{j=1}^{k}\|u_{j}\|$. In particular, $|f(u_{j})|\le B$.

  Set $\eta=\min\{\frac{\varepsilon}{4kB}, \nu_{1},\dots, \nu_{k}\} >0$ and choose
  integers $d\ge 1$ and $a_{j}\ge 1$, $j=1,\dots,k$, such that
  \begin{equation}\label{eq:10}
    \frac{\eta}{2}<\nu _{j}-\frac{a_{j}}{d}<\eta.
  \end{equation}
  Put
  \begin{displaymath}
a_{k+1}=d-\sum_{j=1}^{k}a_{j}, \quad   u_{k+1}=\frac{d}{a_{k+1}}\bigg(u
      -\sum_{j=1}^{k}\frac{a_{j}}{d}u_{j}\bigg),     
  \end{displaymath}
  so that $\sum_{j=1}^{k+1}\frac{a_{j}}{d}=1$ and $
  \sum_{j=1}^{k+1}\frac{a_{j}}{d}u_{j}=u$ holds. Moreover, the
  inequalities in \eqref{eq:10} imply that $
  \frac{k\eta}{2}<\frac{a_{k+1}}{d}<k\eta $ and
  \begin{displaymath}
  \|u_{k+1}\|\le\bigg\|\frac{d}{a_{k+1}}\bigg(u-\sum_{j=1}^{k}
        \frac{a_{j}}{d}u_{j}\bigg)\bigg\|\le
    \frac{2}{k\eta}\eta\sum_{j=1}^{k}\|u_{j}\|\le 2
    \sum_{j=1}^{k}\|u_{j}\|. 
  \end{displaymath}
  Thus $|f(u_{k+1})|\le B$. Now we compute
  \begin{multline*}
    \bigg|\sum_{j=1}^{k}\nu
      _{j}f(u_{j})-\sum_{j=1}^{k+1}\frac{a_{j}}{d}f(u_{j})
    \bigg|\le \bigg|\sum_{j=1}^{k}\bigg(\nu
        _{j}-\frac{a_{j}}{d}\bigg)f(u_{j})-\frac{a_{k+1}}{d}f(u_{k+1})
    \bigg|\\ \le \eta k B+\eta k B\le \varepsilon /2.
  \end{multline*}
Combining this with  \eqref{eq:57},
  \begin{displaymath}
    \conc(f)(u)\le \frac{1}{d}\sum_{j=1}^{k+1}a_{j}f(u_{j})+\varepsilon,
  \end{displaymath}
which proves  the result.
\end{proof}

\begin{rem}
In the previous lemma, the hypothesis that $f$ is locally bounded
below is necessary because there exist non-concave functions that satisfy the
concavity condition for rational convex combinations. For instance, a discontinuous $\Q$-linear function
from $\R$ to $\R$ is not concave because it is not continuous, but it
satisfies the concavity condition for rational combinations because it is
$\Q$-linear. 

The condition of being locally bounded below is trivially satisfied if
$f$ is continuous.  
\end{rem}

The next result gives a criterion for the stability of a conic function to
be nonempty and, \emph{a fortiori}, for the  existence of its
concavification.

\begin{lem} \label{lemm:4} Let $\Psi\colon N_{\R}\to \R$ be a conic
  function. Then 
  $\stab(\Psi)\not =\emptyset$ if and only if, for all collections of points $u_{j}\in N_{\R}$,
  $j=1,\dots,\ell$, such that 
  \begin{math}
\sum_{j=1}^{\ell}u_{j}=0,
  \end{math} we have
  \begin{equation} \label{eq:13}
    \sum_{j=1}^{\ell}\Psi(u_{j})\le 0.
  \end{equation}
\end{lem}

\begin{proof}
Suppose that, for all zero-sum families of points $u_{j}\in N_{\R}$,
  $j=1,\dots,\ell$, the inequality \eqref{eq:13} holds.
For $u\in N_{\R}$, set 
\begin{equation} \label{eq:34}
  \Phi(u)=\sup \sum_{j=1}^{\ell}\Psi(w_{j}) \in\R\cup\{\infty\}, 
\end{equation}
where the supremum is over all  $w_{j}\in N_{\R}$,
  $j=1,\dots,\ell$, such that $u=\sum_{j}w_{j}$.
By~\eqref{eq:13} applied to the points $-u$ and $w_{j}$,
$j=1,\dots, \ell$, 
  \begin{displaymath}
\sum_{j=1}^{\ell}\Psi(w_{j}) \le -\Psi(-u)
  \end{displaymath}
  and so the supremum in \eqref{eq:34} is finite.  Hence,
  \eqref{eq:34} defines a conic function $\Phi\colon N_{\R}\to \R$.
  By construction, $\Phi$ is concave and $\Phi\ge \Psi$. Hence,
  $\stab(\Psi)\supset \stab(\Phi) \ne \emptyset$.

  Conversely, assume that $\stab(\Psi)\not = \emptyset$. Let $x\in
  \stab(\Psi)$. Since $\Psi$ is conic, we have $\Psi(u)\le
  \left<x,u\right>$. Thus
  \begin{displaymath}
    \sum_{j=1}^{\ell}\Psi(u_{j})\le
    \sum_{j=1}^{\ell}\langle x,u_{j}\rangle=0. 
  \end{displaymath}
\end{proof}

\begin{lem}\label{lemm:5}
  Let $f\colon N_{\R}\to \R$ be a function with $\stab(f)\not =
  \emptyset$ and $g\colon N_{\R}\to \R$ another function with $|f-g|$
  bounded. Then $\stab(g)\not = \emptyset$ and $|\conc(f)-\conc(g)|$ is
  bounded. 
\end{lem}
\begin{proof}
  If $|f-g|$ is bounded, then $\stab(f)=\stab(g)$, which gives the
  first statement. Since $|f-g|$ is bounded we can choose $B>0$ such
  that $|f(u)-g(u)|\le B$ for all $u\in N_{\R}$. Fix a point $u\in
  N_{\R}$ and consider a convex combination of points of~$N_{\R}$
  \begin{displaymath}
    u=\sum_{j=1}^{\ell}\nu _{j}u_{j}.
  \end{displaymath}
Then
  \begin{displaymath}
    \sum_{j=1}^{\ell} \nu _{j}f(u_{j})-\conc(g)(u)\le
    \sum_{j=1}^{\ell} \nu _{j}f(u_{j})-\sum_{j=1}^{\ell} \nu
    _{j}g(u_{j})\le B.
  \end{displaymath}
  Since this is true for any convex combination as above, we deduce
  \begin{displaymath}
    \conc(f)(u)-\conc(g)(u)\le B.
  \end{displaymath}
  By symmetry $\conc(g)(u)-\conc(f)(u)\le B$ and the second statement
  follows. 
\end{proof}

\section{Successive minima of toric metrized $\R$-divisors}
\label{sec:succ-algebr-minima}

In this section, we give the formulae for the successive minima of the
height function  associated to a toric metrized $\R$-divisor on a
proper toric variety over a global field. We will use the notations
and results in \cite{BurgosPhilipponSombra:agtvmmh,
  BurgosMoriwakiPhilipponSombra:aptv} although, for the convenience of
the reader, we recall below some of them.

Let $\K$ be a global field as in the previous section and $X$ a
variety over $\K$, that is, a reduced and irreducible separated scheme
of finite type over $\K$.  The elements of $X(\ov \K)$ will be called
the \emph{algebraic points} of $X$. For each place $v\in
\mathfrak{M}_{\K}$, we denote by $X_{v}^{\an}$ the $v$-adic
analytification of $X$. If $v$ is Archimedean, this is is a complex
space (equipped with an anti-linear involution if $\K_{v}\simeq \R$)
and, if $v$ is non-Archimedean, it is a Berkovich space.

Given a (quasi-algebraic)
metrized $\R$-divisor $\ov D$ on $X$ as in \cite[Definition
3.3]{BurgosMoriwakiPhilipponSombra:aptv}, we consider the associated
height function
\begin{displaymath}
  \h_{\ov D}\colon X(\ov \K)\to \R
\end{displaymath}
 defined as follows. 

For each $p\in X(\ov \K)$ choose a function $f\in {\rm
  K}(X)^{\times}_{\R}=\KK(X)^{\times}\otimes \R$ such that $p\not \in
|D-\div(f)|$, the support of $D-\div(f)$. For instance, when $D$ is a
Cartier divisor, we can take $f$ as a local equation of $D$ at $p$.

Choose a finite extension $\F$ of $\K$ such that $p\in X(\F)$. To $f$,
we can associate a metrized $\R$-divisor $\widehat \div(f)$ and we
consider the metrized $\R$-divisor $\ov D-\wh\div(f)$ on $X$. For
simplicity, we also denote by $\ov D-\wh\div(f)$ the metrized
$\R$-divisor on $X_{\F}$ obtained by base change.  To each place $w\in
\mathfrak{M}_{\F}$ can associate a $w$-adic Green function
$$g_{\ov D-\wh \div (f),w}\colon (X^{\an}_{\F})_{w}\setminus
|D-\div(f)|\to \R,
$$
see \cite[Definitions 3.3 and
3.4]{BurgosMoriwakiPhilipponSombra:aptv}. For instance, if $\ov D$ is
a metrized Cartier divisor on $X$ and $p\notin |D|$, we have that
$g_{\ov D,w}(p)= -\log\|s_{D}(p)\|_{w}$ with $s_{D}$ the canonical 
rational section of the line bundle $ \cO(D)$ and $\|\cdot\|_{w}$ the $w$-adic
metric on $\cO(D)_{w}^{\an}$ obtained from the extension of $\ov D$ on
$X_{\F }$ by base change. We denote by $\iota_{w}\colon
X(\F){\rightarrow} (X_{\F})^{\an}_{w}$ the inclusion of the
$\F$-rational points of $X$ into the $v$-adic analytification.

\begin{defn}\label{def:4} With the previous notations, the
  \emph{height} of $p$ with respect to $\ov D$ is given by
  \begin{displaymath}
    \h_{\ov D}(p)= \sum_{w\in \mathfrak{M}_{\F}}n_{w} g_{\ov
      D-\wh \div (f),w}(\iota_{w}(p)). 
  \end{displaymath}
\end{defn}

The height is independent of the choice of the rational function
$f$ and of the extension $\F$.

\begin{rem} \label{rem:1}
This definition is the natural extension to metrized
  $\R$-divisor of the height functions of points from Arakelov
  geometry as in \cite{BostGilletSoule:HpvpGf,
    Zhang:_small,Gubler:lchs,ChambertLoir:meeB,
    BurgosPhilipponSombra:agtvmmh}. Observe that, to define the height
  of cycles of arbitrary dimension in
  \cite{BurgosPhilipponSombra:agtvmmh}, we need the variety to be
  proper and the metrics to be DSP, but these conditions are not
  needed in the case of points. The reason is that Definition
  \ref{def:4} is equivalent to first restricting the metrized divisor to
  the point and then computing the height of the point with respect to
  this restriction, together with the observation that a point is
  proper and that every metric on a point is semipositive.
\end{rem}

Instead of choosing a finite extension where the point $p$ is defined,
we can express the height of an algebraic point in terms of its Galois
orbit. For each place $v$, we choose an arbitrary inclusion
$\jmath\colon \ov
\K\hookrightarrow \ov \K_{v}$. This inclusion induces a map $X(\ov
\K)\hookrightarrow X(\ov \K_{v})$, that we also denote by
$\jmath$. Let 
\begin{equation}\label{eq:35}
X(\ov \K_{v})\overset{\iota}{\longrightarrow} X_{\ov \K_{v}}^{\an}\overset{\pi}{\longrightarrow}
X^{\an}_{v},
\end{equation}
be the maps induced from the extension of valued fields
$\K_{v}\hookrightarrow \ov \K_{v}$, see for instance
\cite[\S~1.2]{BurgosPhilipponSombra:agtvmmh} for the non-Archimedean
case. Consider then the composition
\begin{displaymath}
\varphi_{v}=\pi\circ\iota\circ \jmath \colon X(\ov \K)\to
X^{\an}_{v}.   
\end{displaymath}
Let $G_{\K}=\Aut(\ov \K/\K)$ be the absolute Galois group of $\K$ and
$G_{\K}\cdot p$ the Galois orbit of $p$. The image
$\varphi_{v}(G_{\K}\cdot p)$ of the Galois orbit of $p$ in
$X^{\an}_{v}$ does not depend on the choice of the inclusion $\jmath$
and will be denoted by $(G_{\K}\cdot p)_{v}$.

\begin{prop}\label{prop:4} With the previous hypothesis and
  notation, the height of $p$ with respect to $\ov D$ is given by
  \begin{displaymath}
    \h_{\ov D}(p)=\sum_{v\in \mathfrak{M}_{\K}}\frac{n_{v}}{\# (G_{\K}\cdot p)_{v}}
    \sum_{q\in (G_{\K}\cdot p)_{v}}g_{\ov
      D-\wh \div(f),v}(q),
  \end{displaymath}
  with $f\in {\rm
  K}(X)^{\times}_{\R}$ such that $p\notin |D-\div(f)|$. 
\end{prop}

\begin{proof} Replacing $\ov D$ by $\ov D-\wh \div (f)$, we may assume
  that $p\not\in |D|$. Choose then a finite normal extension $\F$ of
  $\K$ such that $p\in X(\F)$.  Similarly as in~\eqref{eq:35}, for
  each $v\in \mathfrak{M}_{\K}$ and $w\in \mathfrak{M}_{\F}$ with
  $w\mid v$, there are maps
\begin{displaymath}
  X(\F)\overset{\iota_{w}}{\longrightarrow} (X_{\F})^{\an}_{w}
  \overset{\pi _{w}}{\longrightarrow} X^{\an}_{v},
\end{displaymath}
and the corresponding Green functions of $\ov D$ verify that $g_{\ov
  D,w}=g_{\ov D,v}\circ \pi _{w}$. 

Write $G=\Aut(\F,\K)$ and let $\F^{G}$ be the fixed field. Then
$\F/\F^{G}$ is a Galois extension with Galois group $G$ and
$\F^{G}/\K$ is purely inseparable.  Hence, for $v\in
\mathfrak{M}_{\K}$,
  \begin{displaymath}
    \frac{[\F_{w}:\K_{v}]}{[\F:\K]}=
    \frac{[\F_{w}:(\F^{G})_{v}]}{[\F:\F^{G}]}=\frac{1}{\#\mathfrak{M}_{\F,v}},
  \end{displaymath}
where $\mathfrak{M}_{\F,v}$ denotes the set of places of $\mathfrak{M}_{\F}$ over $v$. 
Then, from the definition of the height
of $p$ in Definition \ref{def:4} and the weights of $\F$ in
\eqref{eq:26}, it follows that
\begin{multline}
  \label{eq:43}
    \h_{\ov D}(p)=\sum_{v\in \mathfrak{M}_{\K}}n_{v}\sum_{w\mid v}
\frac{[\F_{w}:\K_{v}]}{[\F:\K]} g_{\ov D,v}(\pi
    _{w}(\iota_{w}(p)))\\=\sum_{v\in \mathfrak{M}_{\K}}\frac{n_{v}}{\# \mathfrak{M}_{\F,v}} \sum_{w\mid v}g_{\ov D,v}(\pi
    _{w}(\iota_{w}(p))).
\end{multline}

The group $G$ acts on $X(\F)$,  on $\mathfrak{M}_{\F,v}$ and on
$(G_{\K}\cdot p)_{v}$, since $p$ is defined over~$\F$. Both
actions are compatible with the previous maps: for each $\gamma \in
G$, $q\in X(\F)$ and $w\mid v$,
\begin{displaymath}
  \pi _{w}(\iota_{w}(\gamma q))=\pi _{\gamma ^{-1}w}(\iota_{\gamma
    ^{-1}w}(q)). 
\end{displaymath}
Furthermore, the action of $G$ on $\mathfrak{M}_{\F,v}$ is
transitive. Hence, the map
\begin{equation*}
  \mathfrak{M}_{\F,v}\longrightarrow (G_{\K}\cdot p)_{v}, \quad
w \longmapsto \pi _{w}(\iota_{w}(p))
\end{equation*}
is surjective and equivariant with respect to the action of $G$. We
deduce that all the fibers of this map have the same
cardinality. Hence, 
\begin{equation*}
\frac{1}{\# \mathfrak{M}_{\F,v}} \sum_{w\mid v}g_{\ov D,v}(\pi
    _{w}(\iota_{w}(p)))  = \frac{1}{\# (G_{\K}\cdot p)_{v}}
    \sum_{q\in (G_{\K}\cdot p)_{v}}g_{\ov
      D,v}(q).
\end{equation*}
The statement follows from this together with \eqref{eq:43}.
\end{proof}

\begin{defn}
  \label{def:1} Let $X$ be a variety over $\K$ and $W\subset X$ a
  locally closed subset.  For $\eta\in \R$, consider the subset of
  algebraic points of $W$ given by
\begin{displaymath}
  W(\ov \K)_{\le \eta}=\{p\in W(\ov \K)\mid \h_{\ov D}(p)\le \eta\}.
\end{displaymath}
Let $d=\dim(W)$. For $i=1,\dots,d+1$, the \emph{$i$-th
  successive minimum} of $W$ with
respect to $\ov D$ is defined as
\begin{displaymath}
  \upmu^{i}_{\ov D}(W)=\inf\big\{\eta\in\R \mid \dim\big(\ov{  W(\ov \K)_{\le
      \eta}}\big)\ge d-i+1\big\}.
\end{displaymath}
We set $\upmu^{\abs}_{\ov D}(W)= \upmu^{d+1}_{\ov D}(W)$ and
$\upmu^{\ess}_{\ov D}(W)= \upmu^{1}_{\ov D}(W)$ for the
\emph{absolute minimum} and the \emph{essential minimum} of $W$ with
respect to $\ov D$, respectively.
\end{defn}

Clearly, 
\begin{equation} \label{eq:2}
 \upmu^{\ess}_{\ov D}(W)=\upmu^{1}_{\ov D}(W)\ge \upmu^{2}_{\ov D}(X)\ge \dots
\ge   \upmu^{d+1}_{\ov
    D}(W)=\upmu^{\abs}_{\ov D}(W).
\end{equation}

The following result shows that the successive minima are stable with
respect to finite maps. 

\begin{prop}\label{prop:9}
  Let $f\colon X\to Y$ be a dominant morphism of varieties over
  $\K$ and $\ov D$ a metrized $\R$-divisor on $Y$.
  \begin{enumerate}
  \item \label{item:8} If $f$ is generically finite then $\upmu_{f^{\ast}\ov
      D}^{\ess}(X)=\upmu_{\ov D}^{\ess}(Y)$.
  \item \label{item:3} If $f$ is finite then $\upmu_{f^{\ast}\ov
      D}^{i}(X)=\upmu_{\ov D}^{i}(Y)$ for $i=1, \dots, \dim (Y)+1 $.
  \end{enumerate}
\end{prop}
\begin{proof}
For   $p\in X(\ov \K)$, the equality $    \h_{f^{\ast}\ov D}(p)=\h_{\ov D}(f(p))$
  holds, see \cite[Theorem 1.5.11(2)]{BurgosPhilipponSombra:agtvmmh}
  and Remark \ref{rem:1}.  It follows that, for any
  real number $\eta $, we have $X(\ov \K)_{\le \eta }=f^{-1} Y(\ov
  \K)_{\le \eta }$.

  We first prove \eqref{item:3}. Being finite, the
  morphism $f$ is proper and, since it is dominant, it is also
  surjective. Hence 
  $Y(\ov \K)_{\le \eta}= f(X(\ov \K)_{\le \eta})$.  

  Now let $1\le i\le \dim(Y)+1$ and suppose that $\upmu_{\ov
    D}^{i}(X)>\eta$. Then there exists a closed subset $V\subset X$
  of dimension bounded by $n-i+1$ and containing $X(\ov \K)_{\le
    \eta}$. The image $f(V)$ is a closed subset of dimension
  bounded by $n-i+1$ and containing $Y(\ov \K)_{\le \eta}$. Hence,
  $\upmu_{\ov D}^{i}(Y)>\eta$ and, since this holds for all real
  numbers below the $i$-th minimum of $X$, it follows that $\upmu_{\ov
    D}^{i}(X)\le \upmu_{\ov D}^{i}(Y)$.

  Conversely, suppose that $\upmu_{\ov D}^{i}(Y)>\eta$ and let
  $W\subset Y$ be a closed subset of dimension bounded by $n-i+1$
  which contains $Y(\ov \K)_{\le \eta}$. Since $f$ is finite, the
  preimage $f^{-1}(W)$ is a closed subset of dimension bounded
  by $n-i+1$ which contains $X(\ov \K)_{\le \eta}$. Hence, $\upmu_{\ov
    D}^{i}(X)>\eta$ and we conclude that $\upmu_{\ov D}^{i}(X)=
  \upmu_{\ov D}^{i}(Y)$.

The statement \eqref{item:8} follows from \eqref{item:3} by restricting
$f$ to  open dense subsets of $X$ and $Y$ where it is finite. 
\end{proof}

We now specialize to the toric case.  
Let $\T\simeq \G_{m}^{n}$ be a split torus of dimension $n$ over $\K$. Let
$N=\Hom(\G_{m},\T)$ be the 
lattice of cocharacters of $\T$, $M=\Hom(\T,\G_{m})=N^{\vee}$ the 
lattice of characters, and write $N_{\R}=N\otimes \R$ and
$M_{\R}=M\otimes \R$. 

Let $X$ be a proper toric variety over $\K$ with torus $\T$, described
by a complete fan $\Sigma$ on $N_{\R}$. Recall that, to each cone
$\sigma \in \Sigma $ correspond an open affine subset
$X_{\sigma }$ and an orbit $O(\sigma )$. In particular, for $\sigma
=\{0\}$ we obtain the principal open subset $X_{0}$ that, in this case,
agrees with the orbit $O(0)$. It is canonically isomorphic to the
split torus $\T$ that acts on the toric variety $X$. The action of
$\T$ on $X$ will be denoted by $(t,p)\mapsto t\cdot p$.

A toric $\R$-divisor on $X$ is an $\R$-divisor invariant under the
action of $\T$. Such a divisor $D$ defines a function $\Psi_{D}\colon
N_{\R}\to \R$ whose restriction to each cone of the fan~$\Sigma$ is
linear, and which is called a ``virtual support function''.  The toric
$\R$-divisor $D$ is nef if and only if $\Psi_{D}$ is concave.  One can
also associate to $D$ the subset $\Delta_{D}\subset M_{\R}$ given as
$\Delta_{D}=\stab(\Psi_{D})$, the stability set of $\Psi_{D}$ as in
\eqref{eq:50}. If $D$ is pseudo-effective, $\Delta_{D}$ is a polytope
and, otherwise, it is the empty set.

For each place $v\in \mathfrak{M}_{\K}$, we associate to the 
torus $\T$ an analytic space $\T^{\an}_{v}$ and we denote by
$\SS^{\an}_{v}$ its compact subtorus. In the Archimedean case, it is
isomorphic to~$(S^{1})^{n}$. In the non-Archimedean case, it is a
compact analytic group, see \cite
[\S~4.2]{BurgosPhilipponSombra:agtvmmh} for a description.  Then, a
metrized $\R$-divisor $\ov D$ on $X$ is \emph{toric} if $D$ is a toric
$\R$-divisor and  its $v$-adic Green function $g_{\ov D,v}$ is invariant with respect
to the action of~$\SS^{\an}_{v}$ or, equivalently, if its 
 $v$-adic metric $\|\cdot\|_{v}$  is invariant with respect
to the action of~$\SS^{\an}_{v}$, for all $v$.

A toric metrized $\R$-divisor $\ov D$ on $X$ defines an adelic family of
continuous functions $\psi_{\ov D,v}\colon N_{\R}\to \R$ indexed by the places of
$\K$. For $v\in \mathfrak{M}_{\K}$, this function is given, for $p\in \T^{\an}_{v}$, by
\begin{equation}\label{eq:45}
  \psi_{\ov D,v}(\val_{v}(p))= \log \|s_{D}(p)\|_{v},
\end{equation}
where $\val_{v}$ is the valuation map in \eqref{eq:44} and $s_{D}$ is
the canonical rational $\R$-section of $D$ as in
\cite[\S~3]{BurgosMoriwakiPhilipponSombra:aptv}.

The  family of functions associated to $\ov D$ satisfies that, for all $v\in
\mathfrak{M}_{\K}$, the function $|\psi_{\ov D,v}-\Psi_{D}|$ is
bounded and, for all $v$ except for a finite number, $\psi_{\ov
  D,v}=\Psi_{D}$.  In particular, the stability set of $\psi_{\ov
  D,v}$ coincides with $\Delta_{D}$.  The toric metrized $\R$-divisor
$\ov D$ is semipositive if and only if $\psi_{\ov D,v}$ is concave
for all $v$.

\begin{exmpl}\label{exm:6} Let $X$ be a proper toric variety over $\K$
  and $D$ a toric $\R$-divisor on~$X$.  The \emph{canonical metric}
  on $D$ is the metric defined, for each $v\in
  \mathfrak{M}_{\K}$ and $p\in
  \T^{\an}_{v}$, by
\begin{displaymath}
\log \|s_{D}(p)\|_{\can,v}=  \Psi_{D}(\val_{v}(p)), 
\end{displaymath}
see \cite[Proposition-Definition
4.3.15]{BurgosPhilipponSombra:agtvmmh}.  We denote the resulting toric
metrized $\R$-divisor by $\ov D^{\can}$. In this case, $\psi_{\ov
  D^{\can},v}=\Psi_{D}$ for all $v$.  In particular, $\ov D^{\can}$ is
semipositive if and only if $D$ is nef.
\end{exmpl}

For each $v\in \mathfrak{M}_{\K}$, we consider the \emph{local roof
  function} $ \vartheta_{\ov D,v}\colon \Delta_{D}\to \R$ that is
given, for $x\in \Delta_{D}$, by
\begin{displaymath}
  \vartheta_{\ov D,v}(x)= \psi_{\ov D,v}^{\vee}(x)= \inf_{u\in
    N_{\R}}(\langle x,u\rangle
  -\psi_{\ov D,v}(u)).
\end{displaymath}
When $\psi_{\ov D,v}$ is concave, the function $ \vartheta_{\ov D,v}$
coincides with the Legendre-Fenchel dual of $\psi_{\ov D,v}$.  This
gives an adelic family of continuous concave functions on $\Delta_{D}$
which are  zero except for a finite
number of places.

The \emph{global roof function} $\vartheta_{\ov D}\colon \Delta_{D}\to
\R$ is defined as the weighted sum 
\begin{displaymath}
  \vartheta_{\ov D}=\sum_{v\in\mathfrak{M}_{\K}}n_{v}\vartheta_{\ov D,v}.
\end{displaymath}
As it is customary in convex analysis, we can also consider
$\vartheta_{\ov D}$ as a function from the whole of $M_{\R}$ to the
extended real line $\R\cup 
\{-\infty\}$ by writing $\vartheta_{\ov D}(x)=-\infty$ for  $x\not 
\in \Delta $. With this convention, $D$ is not pseudo-effective if and
only if $\vartheta_{\ov D}\equiv -\infty$ on $M_{\R}$. 

In the case of toric varieties, the height of an algebraic point can
be expressed in terms of the family of functions $\{\psi_{\ov
  D,v}\}_{v\in \mathfrak{M}_{\K}}$.  Let $p$ is an algebraic point in
the principal open subset $X_{0}$. Since $D$ is a toric $\R$-divisor,
$p$ is not in the support of $D$.  Choose a finite extension $\F$ of
$\K$ such that $p\in X_{0}(\F)$. For simplicity, we will also denote
by $\ov D$ the toric metrized $\R$-divisor on $X_{\F}$ obtained by
base change. Then, by the definition of the height and the definition
of these functions in \eqref{eq:45},
\begin{multline}\label{eq:42}
  \h_{\ov D}(p)=- \sum_{w\in \mathfrak{M}_{\F}} n_{w}\log\| s_{D}(p)\|_{w}
\\  = -\sum_{w\in \mathfrak{M}_{\F}} n_{w}\psi_{\ov D, w}(\val_{w}(p))  =
-\sum_{v\in \mathfrak{M}_{\K}}\sum_{w\mid v} n_{w}\psi_{\ov D, v}(\val_{w}(p)),
\end{multline}
since $\psi_{\ov D,w}=\psi_{\ov D,v}$ for all $w\mid v$
\cite[Proposition 4.3.8]{BurgosPhilipponSombra:agtvmmh}.

\medskip
The following is the key technical result to study successive minima of
toric varieties.

\begin{thm} \label{thm:1}
Let $X$ be a proper toric variety over $\K$ and  $\ov D$  a
toric metrized $\R$-divisor on $X$. Then
\begin{displaymath}
  \upmu^{\abs}_{\ov D}(X_{0}) = \max_{x\in
    M_{\R}}\vartheta_{\ov D}(x).
\end{displaymath}
\end{thm}
\begin{proof}
We first show that
\begin{equation}
  \label{eq:14}
  \upmu^{\abs}_{\ov D}(X_{0}) \ge \max_{x\in
    M_{\R}}\vartheta_{\ov D}(x).
\end{equation}

For shorthand we write $\Psi =\Psi _{D}$, $\Delta =\Delta _{D}$, $\psi
_{v}=\psi _{\ov 
  D,v}$ and $\vartheta _{v}=\vartheta _{\ov D,v}$.
Let $p$ be an algebraic point of $X_{0}$ and choose a finite extension
$\F$ of $\K$ such that $p\in
X_{0}(\F)$.
We have $\sum_{w\in \mathfrak{M}_{\F}} n_{w}\val_{w}(p)=0$ and recall that,
for each $w\in \mathfrak{M}_{\F}$ and $x\in \Delta$, 
\begin{displaymath}
  \vartheta_{w}(x)=\inf_{u\in N_{\R}} \left(\langle x,u\rangle-
  \psi_{w}(u)\right).
\end{displaymath}
Hence, by \eqref{eq:42} and \eqref{eq:46}, for any
$x\in \Delta$, 
\begin{multline*}
  \h_{\ov D}(p) =-\sum_{w} n_{w}
  \psi_{w}(\val_{w}(p))\\
  = \sum_{w} n_{w}(\langle x,\val_{w}(p)\rangle-
  \psi_{w}(\val_{w}(p)))
  \ge \sum_{w}n_{w} \vartheta_{w}(x)=
  \vartheta_{\ov D}(x). 
\end{multline*}
We conclude that $\upmu^{\abs}_{\ov D}(X_{0})\ge \vartheta_{\ov D}(x)$ for all $x\in
\Delta$. Since $\upmu^{\abs}_{\ov D}(X_{0})\ge 
-\infty=\vartheta_{\ov D}(x)$ for all $x\not \in \Delta $, we obtain the
inequality \eqref{eq:14}. 

We now prove
\begin{equation}
  \label{eq:15}
 \upmu^{\abs}_{\ov D}(X_{0}) \le \max_{x\in
    M_{\R}}\vartheta_{\ov D}(x). 
\end{equation}
Let $S$ be a nonempty subset of places $v\in \mathfrak{M}_{\K}$ that
contains all Archimedean places and all places $v$ such that $
\psi_{v}\not\equiv\Psi$. In particular, $S$ contains all the places
where $\vartheta_{v}\not\equiv 0$.

Suppose first that $D$ is pseudo-effective. By \cite[Proposition
4.9(2)]{BurgosMoriwakiPhilipponSombra:aptv}, this is equivalent to the
fact that $\Delta\ne\emptyset$.  By Lemma \ref{lemm:2}, this implies
that $\conc(\psi _{v})$
exists for all $v$. 

Let $x_{0}\in \Delta$ such that $\vartheta_{\ov D}(x_{0})= \max_{x\in
  \Delta}\vartheta_{\ov D}(x)$.
By \cite[Theorem 23.8]{Rockafellar:ca},
\begin{displaymath}
 0\in \sum_{v\in S}n_{v} \partial \vartheta_{v}(x_{0}).
\end{displaymath}
Choose a collection $u_{v}\in N_{\R}$, $v\in S$,  such that 
\begin{displaymath}
u_{v}\in \partial \vartheta_{v}(x_{0})\quad \text{ and } \quad \sum_{v\in
  S}n_{v} u_{v}=0.   
\end{displaymath}
For $v\not \in S$ put $u_{v}=0$.
Since $\vartheta_{v}= \psi_{v}^{\vee}=
\conc(\psi_{v})^{\vee}$ and $\conc(\psi_{v})$ is
concave,
\begin{equation}\label{eq:7}
  \vartheta_{v}(x_{0})= \langle x_{0},u_{v}\rangle
  -\conc(\psi_{v})(u_{v}). 
\end{equation}

Let $\varepsilon >0$. Using Lemma \ref{lemm:6}, we deduce that there
exists $d\ge 1$ and, for all $v\in S$, there exists  
$u_{v,j}\in N_{\R}$, $j=1,\dots, d$, such that
\begin{displaymath}
  \frac{1}{d}\sum_{j=1}^{d}u_{v,j}=u_{v}\quad \text{ and } \quad
  \conc(\psi_{v})(u_{v})\le   \frac{1}{d}\sum_{j=1}^{d}
  \psi_{v}(u_{v,j})+\frac{\varepsilon}{2\sum_{v\in
      S}n_{v}}.
\end{displaymath}
For $v\not \in S$, put $u_{v,j}=0$, $j=1,\dots,d$.
Since for $v\notin S$ we have $\psi_v=\Psi$ and hence $\psi_v(0)=\conc(\psi_v)(0)=0$, we deduce 
\begin{equation} \label{eq:5}
\sum_{v\in \mathfrak{M}_{\K}}
\sum_{j=1}^{d}\frac{n_{v}}{d}\psi_{v}(u_{v,j})\ge
 \sum_{v\in \mathfrak{M}_{\K}}n_{v}
 \conc(\psi_{v})(u_{v})-\frac{\varepsilon}{2}.
\end{equation}
Let $\F/\K$ be an extension of degree $d$ such that all places
in $S$ split completely, as given by Lemma \ref{constrextotdec}.  For each
$v\in S$ and $w\in \mathfrak{M}_{\F}$ such that $w\mid v$, we have 
$n_{w}=n_{v}/d$. We  number the places
above a given place $v\in S$ and  write
them as $w(v,j)$, $j=1,\dots,d$.

Let $p=\alpha^{r}=\alpha\otimes r\in\T(\F)\otimes \Q$ with
$\alpha\in \T(\F)$ and $r\in \Q$. 
Such an element  may be
viewed as a point of $\T$ defined over some radical extension of $\F$.
Hence, $\val_w(p)=r\val_w(\alpha)$ is the common value at $p$ of the valuation maps
associated to the places of this extension over $w$. 

Recall that $H_{\F}\subset \bigoplus _{w\in \mathfrak{M}_{\F}}N_{\R}$ is the
hyperplane defined by the equation
\begin{displaymath}
  \sum_{w\in \mathfrak{M}_{\F}}n_{w}z_{w} =0
\end{displaymath}
as in \eqref{eq:16}.

The functions $\Psi $ and $\psi _{v}$ satisfy the hypothesis of Lemma
\ref{lemm:8}.  Hence, we deduce from Lemmas \ref{lemm:7} and
\ref{lemm:8} that there exists $p=\alpha\otimes r\in\T(\F)\otimes \Q$
with $\alpha\in \T(\F)$ and $r\in \Q$ such that $\val_w(p)=0$ for $w$
above $v\notin S$, with $\val_{w(v,j)}(p)$ sufficiently close to
$u_{v,j}$ for all $v\in S$ and $j=1,\dots,d$. Therefore
\begin{equation} \label{eq:6}
\sum_{w\in \mathfrak{M}_{\F}}  n_{w}\psi_{w}(\val_{w}(p))\ge
\sum_{v\in \mathfrak{M}_{\K}}
\sum_{j=1}^{d}\frac{n_{v}}{d}\psi_{v}(u_{v,j})-
\frac{\varepsilon}{2}.
\end{equation}
From \eqref{eq:5} and \eqref{eq:6}, we deduce that
\begin{displaymath}
\h_{\ov D}(p)  = -\sum_{w\in \mathfrak{M}_{\F}}
n_{w}\psi_{w}(\val_{w}(p))\le -\sum_{v\in
\mathfrak{M}_{\K}}n_{v}\conc(\psi_{v})(u_{v})+ \varepsilon.
\end{displaymath}
Using $\sum_{v\in\mathfrak{M}_\K}n_vu_v=0$ and \eqref{eq:7}, we obtain
\begin{displaymath}
  \h_{\ov D}(p)  \le \sum_{v\in
    \mathfrak{M}_{\K}}n_{v}( \langle x_{0},u_{v}\rangle
  -\conc(\psi_{v})(u_{v}))+ \varepsilon 
  =\vartheta_{\ov D}(x_{0})+\varepsilon. 
\end{displaymath}
From this, we deduce that $\upmu^{\abs}_{\ov D}(X_{0}) \le \max_{x\in
  M_{\R}}\vartheta_{\ov D}(x)+\varepsilon$ for all $\varepsilon >0$
proving the inequality \eqref{eq:15} in the case when $D$ is
pseudo-effective.

If $D$ is not pseudo-effective, then $\stab(\Psi)=\emptyset$ and $\max_{x\in M_\R}\vartheta_{\ov D}(x)=-\infty$. By Lemma
\ref{lemm:4}, there exist $u_{j}\in N_{\R}$, $j=1,\dots, \ell$, such
that 
  \begin{displaymath}
\sum_{j=1}^{\ell}u_{j}=0\quad \text{ and }\quad     \sum_{j=1}^{\ell}\Psi(u_{j})>0.
  \end{displaymath}
Using Lemmas \ref{lemm:8} and \ref{lemm:7}, there exists $p=\alpha\otimes r\in
\T(\K)\otimes \Q$ such that
\begin{displaymath}
\eta:=\sum_{v\in \mathfrak{M}_{\K}} n_{v}\Psi(\val_{v}(p))>0 .
\end{displaymath}
For $l\ge 1$ such that $lr\in\N$, we view $p_l:=\alpha\otimes lr$ as a
point of $\T(\K)$. Then  
\begin{displaymath}
\h_{\ov D^{\can}}(p_l) = \sum_{v\in \mathfrak{M}_{\K}}
-n_{v}\Psi(\val_{v}(p_l))= l\sum_{v\in \mathfrak{M}_{\K}}
-n_{v}\Psi(\val_{v}(p))=-l\eta,
\end{displaymath}
where $\ov D^{\can}$ denotes the $\R$-divisor $D$ equipped with the
canonical metric as in Example \ref{exm:6}.  Since the difference
between the functions $\h_{\ov D^{\can}}$ and $\h_{\ov D}$ is bounded,
it follows that $\lim_{l\to\infty }\h_{\ov D}(p_l) = -\infty$. Hence
$\upmu^{\abs}_{\ov D}(X_{0})=-\infty$, which completes the proof.
\end{proof}

\begin{cor} \label{cor:2}
Let $X$ be a proper toric variety over $\K$ and  $\ov D$  a
toric metrized $\R$-divisor on $X$. Then $\upmu^{\abs}_{\ov D}(X_{0})
>-\infty$ if and only if $D$ is pseudo-effective. 
\end{cor}

\begin{proof}
By \cite[Proposition 4.9(2)]{BurgosMoriwakiPhilipponSombra:aptv}, $D$
is pseudo-effective if and only if $\Delta$ is not empty. The result
follows then from Theorem \ref{thm:1}.
\end{proof}

The next lemma shows that the successive minima of a toric variety with
respect to a toric metrized $\R$-divisor can be computed in terms of
the absolute minima of the orbits under the action of $\T$.

\begin{lem} \label{lemm:3} Let $X$ be a proper toric variety over
  $\K$ of dimension $n$ and $\ov D$ a toric metrized $\R$-divisor on
  $X$. 
  \begin{enumerate}
  \item \label{item:4} Let $\sigma\in \Sigma$. Then
    $\upmu_{\ov D}^{i}(O(\sigma))= \upmu_{\ov D}^{\abs}(O(\sigma))$ for  $i=1,\dots,\dim(O(\sigma))+1$.
  \item \label{item:5} For $i=1,\dots, n+1$,
  \begin{displaymath}
\upmu_{\ov D}^{i}(X)= \min_{\sigma\in \Sigma^{\le i-1}} \upmu_{\ov D}^{\abs}(O(\sigma)),    
  \end{displaymath}
where  $\Sigma^{\le i-1} $ denotes the set of cones of $\Sigma$ of
dimension $\le i-1$. 
  \end{enumerate}
\end{lem}

\begin{proof}
  We first prove \eqref{item:4}.  Let $\T(\ov \K)_{\tors}$ denote the
  subgroup of torsion points of the group of algebraic points of
  $\T$. Under the identification $\T(\ov \K)= \Hom(M,\ov \K^{\times})$,
  this subgroup corresponds to $\Hom(M,\upmu_{\infty})$, the
  homomorphisms from $M$ to the group of roots of unity. This implies
  that, if $t\in \T(\ov \K)_{\tors}$ and
  $p\in X(\ov \K)$, then $\h_{\ov D}(t\cdot p)=\h_{\ov D}(p)$.

Now let $\varepsilon >0$ and choose an
algebraic point  $p$ of the orbit $O(\sigma)$ such that $\h_{\ov D}(p)\le
\upmu_{\ov D}^{\abs}(O(\sigma))+\varepsilon$. Then, $\T(\ov \K)_{\tors}
\cdot p$ is a dense subset of algebraic points of $O(\sigma )$ of the
same height as $p$. Hence, 
\begin{displaymath}
\upmu_{\ov D}^{\ess}(O(\sigma))\le 
\upmu_{\ov D}^{\abs}(O(\sigma))+\varepsilon.
\end{displaymath}
We deduce that $\upmu_{\ov D}^{\ess}(O(\sigma))\le 
\upmu_{\ov D}^{\abs}(O(\sigma))$. If follows that the chain of
inequalities in \eqref{eq:2}, applied to the variety $O(\sigma)$,
shrinks to the equalities in the statement.

Now we consider  \eqref{item:5}.
Let $\sigma \in \Sigma^{\le i-1}$. Then, $O(\sigma)$ is of
dimension $\ge n-i+1$ and so 
\begin{displaymath}
  \upmu_{\ov D}^{i}(X)\le \upmu_{\ov D}^{\ess}(O(\sigma)).
\end{displaymath}
Using \eqref{item:4}, we deduce that $\upmu_{\ov D}^{i}(X)\le \min_{\sigma\in
  \Sigma^{\le i-1}} \upmu_{\ov D}^{\abs}(O(\sigma))$.

For the reverse inequality, observe that, for a cone $\sigma \in
\Sigma$ of dimension $\ge i$, the orbit $O(\sigma)$ is of dimension
$\le n-i$. Using the decomposition of $X$ into orbits, we deduce that
\begin{displaymath}
  \upmu_{\ov D}^{i}(X)= \upmu_{\ov D}^{i}\Big(X\setminus \bigcup_{\sigma\in \Sigma^{\ge
      i}}O(\sigma)\Big)= \upmu_{\ov D}^{i}\Big(\bigcup_{\sigma\in \Sigma^{\le
      i-1}}O(\sigma)\Big).
\end{displaymath}
Hence, 
\begin{displaymath}
  \upmu_{\ov D}^{i}(X)\ge \upmu_{\ov D}^{\abs}\Big(\bigcup_{\sigma\in \Sigma^{\le
      i-1}}O(\sigma)\Big) =  \min_{\sigma\in
  \Sigma^{\le i-1}} \upmu_{\ov D}^{\abs}(O(\sigma)),
\end{displaymath}
which proves the result. 
\end{proof}

Theorem \ref{thm:5} in the introduction is a direct consequence of
the previous results.

\begin{cor} \label{cor:4} 
Let $X$ be a proper toric variety over $\K$ and  $\ov D$  a
toric metrized $\R$-divisor on $X$. Then
\begin{displaymath}
  \upmu^{\ess}_{\ov D}(X)=  \max_{x\in
    M_{\R}}\vartheta_{\ov D}(x).
\end{displaymath}
\end{cor}
\begin{proof}
  From  Lemma
  \ref{lemm:3}\eqref{item:5} and Theorem \ref{thm:1} we obtain
  \begin{displaymath}
    \upmu^{\ess}_{\ov D}(X)=  \upmu^{\abs}_{\ov D}(X_{0})=\max_{x\in
    M_{\R}}\vartheta_{\ov D}(x).
  \end{displaymath}
\end{proof}

Using this  result, we can deduce some  relations between the essential
minimum and the positivity properties of $\ov D$ in the toric
setting. 

\begin{cor}\label{cor:7}
Let $X$ be a proper toric variety over $\K$
  of dimension $n$ and $\ov D$ a toric metrized $\R$-divisor on
  $X$. Then 
  \begin{enumerate}
  \item \label{item:7} $\ov D$ is pseudo-effective if and only if
    $\upmu^{\ess}_{\ov D}(X)\ge0$;
  \item \label{item:9} $\ov D$ is big if and only if
    $\dim(\Delta_{D})=n$ and $\upmu^{\ess}_{\ov D}(X) > 0$.
  \end{enumerate}
\end{cor}

\begin{proof}
  This follows from Corollary \ref{cor:4} and \cite[Theorem
  2(3,4)]{BurgosMoriwakiPhilipponSombra:aptv}.
\end{proof}

It is possible to reformulate Corollary \ref{cor:4} to give a formula
for the essential minimum in terms of the functions $\psi _{\ov D,v}$
that is useful when computing the essential minimum in concrete
situations as those considered in \S~\ref{sec:exmpl}.

For convenience we recall the definition of sup-convolution of two
concave functions \cite[\S~2.3]{BurgosPhilipponSombra:agtvmmh}. Let
$\Delta \subset M_{\R}$ be a polytope and $\psi 
_{1}$, $\psi _{2}$ two concave functions on $N_{\R}$ whose stability set is
$\Delta $. Then 
\begin{displaymath}
  (\psi _{1}\boxplus \psi _{2})(u)=
  \sup_{u_{1}+u_{2}=u} \psi _{1}(u_{1})+ \psi _{2}(u_{2}).
\end{displaymath}
This is a concave function on $N_{\R}$ with stability set $\Delta $.
The sup-convolution is an associative operation. In fact it
corresponds to the pointwise addition by Legendre-Fenchel duality
\cite[Proposition 
  2.3.1]{BurgosPhilipponSombra:agtvmmh}. That is,
\begin{displaymath}
  \psi _{1}\boxplus \psi _{2}=(\psi _{1}^{\vee}+
  \psi^{\vee} _{2})^{\vee}.
\end{displaymath}
Moreover, if $\psi_2=\Psi $
is the support 
function of $\Delta $, then $\psi _{1}\boxplus \Psi =\psi _{1}$.

Recall also the right multiplication of a concave function by an
scalar.
\begin{displaymath}
  (\psi _{1}\lambda)(u)=\lambda \psi _{1}(u/\lambda ). 
\end{displaymath}
This operation is dual of the usual left multiplication
\begin{displaymath} 
  \psi _{1}\lambda=(\lambda \psi _{1}^{\vee})^{\vee}.
\end{displaymath}

\begin{cor}\label{cor:8} Let $X$ be a proper toric variety over $\K$
  and  $\ov D$  a
toric metrized $\R$-divisor on $X$. 
\begin{enumerate}
\item \label{item:21} If $D$ is pseudo-effective, then
\begin{displaymath} 
\upmu^{\ess}_{\ov D}(X) = \upmu^{\abs}_{\ov D}(X_{0})=
-\Big(\boxplus_{v\in\mathfrak{M}_\K} \conc(\psi_{\ov D,v}n_v)\Big)(0).
\end{displaymath}
\item \label{item:20} If $\ov D$ is semipositive and its $v$-adic metric
agrees with the canonical
metric for all places except one place $v_{0}$, then 
\begin{equation*}
\upmu^{\ess}_{\ov D}(X) = \upmu^{\abs}_{\ov D}(X_{0}) = -n_{v_{0}}\psi_{\ov D,v_{0}}(0).
\end{equation*}

\end{enumerate}
\end{cor}
\begin{proof}
  By Theorem \ref{thm:1},   $\upmu^{\abs}_{\ov D}(X_{0}) = \max_{x\in
    M_{\R}}\vartheta_{\ov D}(x)$. From the definition of
  Legen\-dre-Fenchel duality, $\max_{x\in
    M_{\R}}\vartheta_{\ov D}(x)=-\vartheta _{\ov D}^{\vee}(0)$. By the
  duality between the
  sum and the sup-convolution, and  that
  between the right and the left multiplication, 
  \begin{displaymath}
    \vartheta _{\ov D}^{\vee}=\boxplus_{v\in\mathfrak{M}_\K} 
    \conc(\psi_{\ov D,v}n_v).
  \end{displaymath}
  Hence we obtain the first statement.

  If $\ov D$ is semipositive and its metric agrees with the canonical
  metric for all places except one place $v_{0}$, then  $\psi_{\ov
    D,v}=\Psi_{D}$ for all $v\ne v_{0}$ and $\psi_{\ov D,v_{0}}$ is
  concave. Since the semipositivity of $\ov D$ implies that $D$ is
  pseudo-effective, the first statement implies that
 \begin{displaymath}
   \upmu^{\ess}_{\ov D}(X) =
    -(\psi_{\ov D,v_{0}}n_{v_{0}})(0)=-n_{v_{0}}\psi_{\ov
      D,v_{0}}(0/n_{v_{0}})=-n_{v_{0}}\psi_{\ov D,v_{0}}(0),
  \end{displaymath}
  which proves the second statement.
\end{proof}

\begin{rem}\label{rem:3}
If $\ov D$ is semipositive and its $v$-adic metric
agrees with the canonical
metric for all places except one place $v_{0}$, then we can identify a
dense set of points whose height agrees with $\upmu^{\ess}_{\ov
  D}(X)$. Namely, each point
\begin{displaymath}
  p\in\T(\ov \K)_{\tors}\subset X_{0}(\ov \K)
\end{displaymath}
satisfies $\val_{v}(p)=0$ for all $v\in \mathfrak{M}_{\K}$. Hence
\begin{displaymath}
  \h_{\ov D}(p)=-\sum_{v\in \mathfrak{M}_{\K}}n_{v}\psi _{\ov D,v}(0)=
  -n_{v_{0}}\psi _{\ov D,v_{0}}(0)=\upmu_{\ov D}^{\ess}(X).
\end{displaymath}

\end{rem}

We now want to extend Corollary \ref{cor:4}
to the other successive
minima. To do this, we have to describe the restriction  to  an orbit of a toric
metrized $\R$-divisor in terms of its roof function.

Let $X$ be a proper toric variety and $\ov
D$ a toric metrized $\R$-divisor with $D$ nef. Let $\Sigma $ be the
fan of $X$ and $\Delta_{D} $ the polytope associated to $D$.
Recall that, as in \cite[Example~2.5.13]{BurgosPhilipponSombra:agtvmmh}, to a
cone $\sigma \in \Sigma $ we can associate a face
$F_{\sigma }\subset \Delta_{D} $ given by
\begin{displaymath}
  F_{\sigma }=\{x\in \Delta_{D} \mid \langle y-x,u\rangle \ge 0\text{ for
    all }y\in \Delta_{D} , u\in \sigma\}.
\end{displaymath}
Choose an element $m_{\sigma}\in M_{\R}$ in the affine space generated by the
face $F_{\sigma}$. Then the $\R$-divisor $D- \div(\chi^{-m_{\sigma
  }})$ intersects the orbit closure $V(\sigma )$ properly, see
\cite[Proposition 3.3.14]{BurgosPhilipponSombra:agtvmmh} for the case
of Cartier divisors. We set
\begin{equation}
  \label{eq:63}
  D_{\sigma}=(D-\div(\chi^{-m_{\sigma
  }}))|_{V(\sigma)} \quad \text{ and } \quad   \ov D_{\sigma}=(\ov D-\wh\div(\chi^{-m_{\sigma
  }}))|_{V(\sigma)}
\end{equation}
with $\wh\div(\chi^{-m_{\sigma }})$ as in \cite[Definition
3.4]{BurgosMoriwakiPhilipponSombra:aptv}.
In this situation, $\Delta_{D_{\sigma}}=F_{\sigma}-m_{\sigma}$.

\begin{prop}\label{prop:1} 
Let $X$ be a proper toric variety and $\ov
D$ a toric metrized $\R$-divisor with $D$ nef. Let $\Sigma $ be the
fan of $X$ and $\sigma \in \Sigma $. With notation as in
\eqref{eq:63}, for all $x\in F_{\sigma}-m_{\sigma}$, 
\begin{displaymath}
\vartheta _{\ov D_{\sigma}}(x) \ge \vartheta _{\ov D}(x+m_{\sigma}).
\end{displaymath}
Moreover, if $\ov D$ is semipositive, the equality holds. 
\end{prop}

\begin{proof} For short, write $\Psi=\Psi_{D}$ and $\psi_{v}=\psi_{\ov
    D,v}$.  Since $D$ is nef, the function $\Psi $ is concave.  Since
  $ m_{\sigma }\in M_{\R}$ belongs to the affine space generated by
  the face $F_{\sigma}$, this implies that $\Psi |_{\sigma }=m_{\sigma
  }$. Let $N_{\sigma }$ be the partial compactification of $N_{\R}$ in
  the direction of $\sigma $
  \cite[(4.1.5)]{BurgosPhilipponSombra:agtvmmh}. Recall that there is
  an inclusion $N(\sigma )_{\R}\subset N_{\sigma }$. The function
  $\Psi -m_{\sigma }$ extends to a continuous function on $N_{\sigma
  }$ and thus can be restricted to $N(\sigma )$. We denote by
  $\Psi({\sigma })$ this restriction.  For each place $v\in
  \mathfrak{M}_{\K}$, the function $\psi _{v}-m_{\sigma }$ can also be
  extended to a continuous function on $N_{\sigma }$ and restricted to
  $N(\sigma )$. We denote by $\psi _{v }(\sigma)$ this restriction.  

  By the analogue of \cite[Proposition
  3.3.14]{BurgosPhilipponSombra:agtvmmh} for $\R$-divisors, the
  support function of $D_{\sigma }$ is $\Psi ({\sigma })$. By the
  commutative diagram in \cite[Proposition
  4.1.6]{BurgosPhilipponSombra:agtvmmh} the metric induced on
  $D_{\sigma }$ by the metric of $D$ is given by the family of
  functions $\{\psi_{v}(\sigma) \}_{v\in \mathfrak{M}_{\K}}$.

  Assume that $\ov D$ is semipositive. Hence for every $v\in
  \mathfrak{M}_{\K}$ the function $\psi _{v}$ is concave. We identify
  $\stab(\Psi ({\sigma }))$ with $F_{\sigma }$ by means 
of the translation by $-m_{\sigma }$. By the
analogue for $\R$-divisors of \cite[Proposition
4.8.9]{BurgosPhilipponSombra:agtvmmh}, we have that, for all $x\in
F_{\sigma}-m_{\sigma}$,
\begin{displaymath}
  \psi _{v}(\sigma )^{\vee}(x) =\vartheta _{\ov D,v}(x+m_{\sigma}).
\end{displaymath}
Summing up for $v\in \mathfrak{M}_{\K}$, we obtain the equality in this case. 

We now drop the hypothesis of $\ov D$ being semipositive. In this
case, the functions $\psi _{v}$ are not necessarily concave.  Since
$D$ is nef, the function $\Psi $ is concave. By Lemma \ref{lemm:5},
for each $v\in \mathfrak{M}_{\K}$, we have that $|\conc(\psi
_{v})-\Psi |$ is bounded. By \cite[Proposition
4.19(1)]{BurgosMoriwakiPhilipponSombra:aptv} the family of functions
$\{\conc(\psi _{v})\}_{v\in \mathfrak{M}_{\K}}$ determines a
semipositive toric metric on $D$. We denote by $\ov D'$ the
corresponding semipositive metrized $\R$-divisor. Since for each
function $f$ with nonempty stability set
\begin{displaymath}
  \conc(f)^{\vee}=f^{\vee},
\end{displaymath}
we have that $\vartheta_{\ov D'}=\vartheta_{\ov D}$.  Since $\psi
_{v}\le\conc(\psi _{v})$ on $N_{\R}$, then $\psi_{v}(\sigma)\le
\conc(\psi_{v})(\sigma)$ on $N(\sigma)_{\R}$.  This implies that
$\vartheta _{\ov D_{\sigma},v}\ge \vartheta _{\ov D'_{\sigma},v}$ on
$F_{\sigma}-m_{\sigma}$ for all $v$. Since $\ov D'$ is semipositive,
it follows that, for all $x\in F_{\sigma}-m_{\sigma}$, 
\begin{displaymath}
\vartheta _{\ov D_{\sigma}}(x)\ge 
\vartheta _{\ov D'_{\sigma}}(x) = \vartheta _{\ov D'}(x+m_{\sigma})= \vartheta _{\ov D}(x+m_{\sigma}),
\end{displaymath}
which concludes the proof.
\end{proof}

\begin{thm}
  \label{thm:2}
Let $X$ be a proper toric variety over $\K$ and  $\ov D$  a
toric metrized $\R$-divisor on $X$ with $D$ nef.
\begin{enumerate}
\item \label{item:22} Let $\Sigma $ be the
fan of $X$ and $\sigma \in \Sigma $. Then
  \begin{equation*}
    \upmu^{\abs}_{\ov D}(O(\sigma ))\ge \max_{x\in F_{\sigma
      }}\vartheta_{\ov D} (x).
  \end{equation*}
If $\ov D$ is semipositive, then the equality holds. 
\item \label{item:23} For $i=1,\dots,n+1$,  
  \begin{equation*}
  \upmu^{i}_{\ov D}(X) \ge \min_{\sigma \in \Sigma ^{\le i-1}} \max_{x\in
    F_{\sigma }}\vartheta_{\ov D}(x).
  \end{equation*}
If $\ov D$ is semipositive, then the equality holds. 
\end{enumerate}
\end{thm}

\begin{proof} The first statement follows directly from Theorem
  \ref{thm:1} and Proposition \ref{prop:1}. The second statement
  follows from the first one together with Lemma
  \ref{lemm:3}\eqref{item:5}.
\end{proof}

\begin{rem} 
  Since $\vartheta_{\ov D} $ is a concave function, its minimum is
  attained at the vertices of $\Delta _{D}$. Therefore, if $\ov D$ is
  semipositive, we deduce 
  \begin{displaymath}
    \upmu^{\abs}_{\ov D}(X) = \min_{x\in
    \Delta _{D}}\vartheta_{\ov D}(x).
  \end{displaymath}
  This result was already implicit in \cite[Theorem
  2(2)]{BurgosMoriwakiPhilipponSombra:aptv} and has been generalized
  by Ikoma to non-toric varieties using Okounkov bodies and concave
  transforms \cite{Ikoma:ncnad}. It would be interesting to know if
  Corollary \ref{cor:4} can also be generalized to non-toric
  varieties.
\end{rem}

When the divisor $D$ is ample, we are able to  give a version of
Theorem \ref{thm:2}\eqref{item:23} in terms the polytope $\Delta_{D}$
only. 

\begin{thm}\label{thm:4}
Let $X$ be a proper toric variety over $\K$ and  $\ov D$  a
semipositive toric metrized $\R$-divisor on $X$ with $D$ ample.
Then, for $i=1,\dots,n+1$,  
  \begin{displaymath}
  \upmu^{i}_{\ov D}(X) = \min_{F\in \cF(\Delta_{D})^{n-i+1}} \max_{x\in
    F}\vartheta_{\ov D}(x),
  \end{displaymath}
  where $\cF(\Delta_{D})^{n-i+1}$ is the set of faces of dimension
  $n-i+1$ of the polytope. 
\end{thm}
\begin{proof}
  If $D$ is ample, the correspondence $\sigma \mapsto F_{\sigma }$ is
  a bijection that sends cones of dimension $i-1$ to faces of
  dimension $n-i+1$. Thus, by Theorem \ref{thm:2}\eqref{item:23}, we deduce
  \begin{displaymath}
  \upmu^{i}_{\ov D}(X) = \min_{F\in \cF(\Delta_{D})^{\ge n-i+1}} \max_{x\in
    F}\vartheta_{\ov D}(x).    
  \end{displaymath}
  where $\cF(\Delta_{D})^{\ge n-i+1}$ is the set of faces of dimension greater
  of equal to $n-i+1$.
  The concavity of $\vartheta _{\ov D}$ implies that the minimum in
  the right hand side is attained in faces of dimension $n-i+1$, hence
  the result.
\end{proof}

\begin{exmpl}\label{exm:1} The positivity conditions on $D$ and $\ov D$ 
  in Theorems \ref{thm:2} and \ref{thm:4} are necessary, 
  as it can be seen in the following examples.
  \begin{enumerate}
  \item Let $X$ be a toric surface over $\K$ and
  let $\ov D$ be a toric metrized $\R$-divisor such that the
  underlying divisor $D$ is big but not nef and that there is a one
  dimensional orbit $O(\sigma )$ such that
  $\deg_{D}(\overline{O(\sigma )})<0$. In this case
  \begin{displaymath}
    \upmu_{\ov{D}}^{2}(X)\le\upmu_{\ov{D}}^{\abs}(O(\sigma ))=-\infty
  \end{displaymath}
  but 
  \begin{displaymath}
     \min_{\sigma\in\Sigma^{\le1}}\max_{x\in F_\sigma}\vartheta_{\ov D}(x) \ge   
     \min_{F\in \cF(\Delta)} \max_{x\in F}\vartheta_{\ov D}(x) > -\infty.
  \end{displaymath}
  Thus the hypothesis $D$ nef in Theorem \ref{thm:2}\eqref{item:23} is necessary.
\item Consider $X=\P^{1}_{\Q}$ and the divisor $\infty = (0:1)$. We
  consider the toric metric given by the canonical metric for $v\not =
  \infty$ and
  \begin{displaymath}
    \psi _{\infty}(u)=
    \begin{cases}
      u&\text{ if }u\le 0,\\
      0&\text{ if }0\le u\le 99, \text{ or } 101\le x, \\
      u-99&\text{ if }99\le u \le 100,\\
      101-u&\text{ if }100\le u \le 101.
    \end{cases}
  \end{displaymath}
  Denote $\ov D$ the obtained metrized $\R$-divisor. The associated
  polytope is $\Delta =[0,1]$, and the roof function is
  \begin{displaymath}
    \vartheta _{\ov D}(x)=
    \begin{cases}
      100x-1&\text{ if }0\le x\le 1/100,\\
      0&\text{ if }1/100\le x \le 1.
    \end{cases}
  \end{displaymath}
  Since $h_{\ov D}(p)\ge \val_\infty(p)-\psi_\infty(\val_\infty(p))$ we deduce $h_{\ov D}(p)\ge0$ for all $p\in X(\ov{\Q})$ and then
  \begin{displaymath}
    \upmu^{1}_{\ov D}(X)=\upmu^{2}_{\ov D}(X)=0,
  \end{displaymath}
  but
  \begin{displaymath}
    \min_{\sigma \in \Sigma ^{\le 1}} \max_{x\in
      F_{\sigma }}\vartheta_{\ov D}(x)=
    \min_{F\in \cF(\Delta)^{0}} \max_{x\in
    F}\vartheta_{\ov D}(x)=-1.
  \end{displaymath}
  Thus, we see that the hypothesis $\ov D$ semipositive is necessary
  for the equality in Theorem \ref{thm:2}\eqref{item:23} to hold.
\item Let $X$ be the blow-up of $\P^{2}_{\Q}$ at the point $(1:0:0)$
  and let $\ov D$ be the preimage of the metrized divisor given by the
  hyperplane at infinity with the canonical metric at the
  non-Archimedan places and the Fubini-Study metric at the Archimedean
  place.   

  Let $\sigma_{0} \in \Sigma $ be the one dimensional cone
  corresponding to the exceptional divisor. Then
  $F_{\sigma _{0}}$ is the vertex $(0,0)$ and has dimension zero. 
  Thus
  \begin{displaymath}
    \upmu^{2}_{\ov D}(X)= \min_{\sigma \in \Sigma ^{\le 1}} \max_{x\in
    F_{\sigma }}\vartheta_{\ov D}(x) =0, 
  \end{displaymath}
  while
  \begin{displaymath}
    \min_{F\in \cF(\Delta)^{1}} \max_{x\in
    F}\vartheta_{\ov D}(x)=\frac{1}{2}\log(2).    
  \end{displaymath}
  Hence the hypothesis $D$ ample is necessary in Theorem \ref{thm:4}.
  \end{enumerate}

\end{exmpl}

\section{On Zhang's theorem on successive
  minima}\label{sec:theor-succ-minima}

Zhang's theorem on successive minima \cite[Theorem 5.2]{Zhang:plbas},
\cite[Theorem 1.10]{Zhang:_small}, shows that the successive minima of
a metrized divisor can be estimated in terms of the height and the
degree of the ambient variety. This result plays an important r\^ole in
Diophantine geometry in the direction of the Bogomolov conjecture and
the Lehmer problem and its generalizations. It also plays a r\^ole in
the study of the distribution of Galois orbits of points of small
height.

We start by giving a proof of a variant of Zhang's theorem in the toric setting
(Theorem \ref{thm:8} in the introduction).

\begin{thm} \label{thm:6} Let $X$ be a proper toric variety over
  $\K$ of dimension $n$ and $\ov D$ a  semipositive toric metrized
  $\R$-divisor on
  $X$ such that $D$ is big. Then
  \begin{equation}\label{eq:25}
\sum_{i=1}^{n+1}\upmu^{i}_{\ov D}(X)\le \frac{\h_{\ov D}(X)}{ \deg_{D}(X)}\le (n+1)
\upmu^{\ess}_{\ov D}(X).
  \end{equation}
\end{thm}

\begin{proof}
  For short write $\Delta=\Delta_{D}$.  Since $\ov D$ is a
  semipositive toric metrized divisor, necessarily $D$ is generated by
  global sections \cite[Corollary
  4.8.5]{BurgosPhilipponSombra:agtvmmh}, hence, being toric
  $\deg_{D}(X)=n!\vol_{M}(\Delta )$. Since $D$ is big, we also have
  $\vol(\Delta )>0$.  We first prove the inequality in the right hand
  side of \eqref{eq:25}. By Corollary \ref{cor:4}, $\upmu^{\ess}_{\ov
    D}(X)\ge \vartheta _{\ov D}(x)$ for all $x\in \Delta $. Therefore,
  by \cite[Theorem 5.2.5]{BurgosPhilipponSombra:agtvmmh}
  \begin{multline*}
        \h_{\ov D}(X)=(n+1)!\int_{\Delta }\vartheta _{\ov
      D}\dd\vol_{M}\le
    (n+1)!\int_{\Delta }\upmu^{\ess}_{\ov D}(X)\dd\vol_{M}\\=
    (n+1)\upmu^{\ess}_{\ov D}(X)n!\vol_{M}(\Delta 
    )=(n+1)\upmu^{\ess}_{\ov D}(X)\deg_{D}(X).
  \end{multline*}
  We now prove the left inequality. For each face $F$ of $\Delta $ we choose
  a point $x_{F}\in F$ such that
  \begin{displaymath}
    \vartheta_{\ov D} (x_{F})=\max_{x\in F}\vartheta_{\ov D} (x). 
  \end{displaymath}
  For each flag of faces
  \begin{displaymath}
    \Xi=\{F_{0}\subsetneq F_{1}\subsetneq\dots \subsetneq F_{n}=\Delta \}
  \end{displaymath}
  with $\dim F_{i}=i$, we denote
  \begin{displaymath}
    \Delta _{\Xi}=\conv(x_{F_{0}},\dots,x_{F_{n}}).
  \end{displaymath}
  Then $\Delta _{\Xi}$ is a (possibly degenerate) simplex. Moreover
  \begin{displaymath}
    \Delta =\bigcup_{\Xi}\Delta _{\Xi} \quad\text{ and }\quad \inter
    \Delta _{\Xi}\cap
    \inter \Delta _{\Xi'}=\emptyset, \text{ if }\Xi\not=\Xi'.
  \end{displaymath}
  Let $f\colon \Delta \to \R$ be the function determined by
  \begin{enumerate}
  \item For each complete flag $\Xi$, the restriction $f\mid_{\Delta _\Xi}$ is
    affine.
  \item If $F$ is a face of $\Delta $ of dimension $i$, then
    $f(x_{F})=\upmu^{n-i+1}_{\ov D}(X)$. 
  \end{enumerate}
  Given a face $F$ of $\Delta $ of dimension $i$, there exists a cone
  $\sigma \in \Sigma ^{n-i}$ such that $F=F_{\sigma }$. Therefore, 
  by Theorem \ref{thm:2}\eqref{item:23}, $$f(x_{F})=\upmu^{n-i+1}_{\ov
    D}(X)\le\max_{x\in F_{\sigma }}\vartheta _{\ov D}(x)=\vartheta _{\ov
    D}(x_{F}).$$ 
  Since $\vartheta _{\ov D}$ is concave
  and $f$ is affine in each simplex $\Delta _{\Xi}$, we deduce $f(x)\le
  \vartheta _{\ov D}(x)$ for all $x\in \Delta $. Therefore
  \begin{displaymath}
    \int_{\Delta }\vartheta _{\ov
      D}\dd\vol_{M} \ge
    \int_{\Delta }f\dd\vol_{M}=
    \sum_{\Xi}\int_{\Delta_{\Xi} }f\dd\vol_{M}.
  \end{displaymath}
Since
\begin{displaymath}
    \sum_{\Xi}\int_{\Delta_{\Xi} }f\dd\vol_{M}
    =\sum_{\Xi}\frac{\sum_{i=0}^{n}\upmu^{n-i+1}_{\ov
        D}(X)}{n+1} \vol (\Delta _{\Xi})
    =\frac{\sum_{i=1}^{n+1}\upmu^{i}_{\ov D}(X)}{n+1}\vol(\Delta),
\end{displaymath}
we deduce
 \begin{displaymath}
    \h_{\ov D}(X)=(n+1)!\int_{\Delta }\vartheta _{\ov
      D}\dd\vol_{M} \ge n!\sum_{i=1}^{n+1}\upmu^{i}_{\ov D}(X)\vol(\Delta
    ) = \sum_{i=1}^{n+1}\upmu^{i}_{\ov D}(X)\deg_{D}(X),
  \end{displaymath}
 proving the result.
\end{proof}

\begin{cor} \label{cor:1}
  Suppose that $\ov D$ is nef. Then 
  \begin{displaymath}
\upmu^{\ess}_{\ov D}(X) \le \frac{\h_{\ov D}(X)}{\deg_{D}(X)}
\le (n+1) \upmu^{\ess}_{\ov D}(X).
  \end{displaymath}
\end{cor}
\begin{proof}
  Since $\ov D$ is nef, all the successive minima are
  non-negative. Then the corollary follows directly from Theorem
  \ref{thm:6}.
\end{proof}

The following result improves \cite[Th\'eor\`eme
1.4]{PhilipponSombra:advtp}. We show that, already for the universal
line bundle on $\P_{\Q}^{n}$, almost every configuration for the successive minima
and the height satisfying the inequalities in \eqref{eq:25}, can be
realized.

\begin{prop} \label{prop:11}
  Let $r\ge 1$ and $\nu, \mu_{1},\dots, \mu_{r+1}\in \R$ such that 
  \begin{displaymath}
    \mu_{1}\ge \dots\ge \mu_{r+1} \quad \text{ and } \quad
    \sum_{i=1}^{r+1}\mu_{i}\le \nu <(r+1) \mu_{1}. 
  \end{displaymath}
Then there exists a semipositive toric metric on $H$, the divisor
given by the hyperplane at infinity of $\P^{r}_{\Q}$, such that 
\begin{displaymath}
  \upmu_{\ov H}^{i}(\P^{r})=\mu_{i},\  i=1,\dots, r+1, \quad \text{ and
  } \quad \h_{\ov H}(\P^{r})=\nu.
\end{displaymath}
\end{prop}

\begin{proof}
  Let $e_{1},\dots,e_{r}$ be the standard basis of $\R^{r}$ and
$\Delta^{r}=\conv(0,e_{1},\dots,e_{r})$ the standard simplex of
$\R^{r}$. For $0\le t <1$ 
consider the function $\theta_{t}\colon \Delta^{r}\to \R$ defined as
the smallest concave function on $\Delta^{r}$ such that 
\begin{displaymath}
  \theta_{t}(x) = 
  \begin{cases}
\mu_{1} & \text{ for } x\in t\Delta^{r}, \\
\mu_{i} & \text{ for } x=e_{i-1} \text{ and } i=2,\dots, r+1.
  \end{cases}
\end{displaymath}
Then the integral $\int_{\Delta^{r}}\theta_{t}\dd x$ varies
continuously in the interval 
\begin{displaymath}
\left[  \frac{1}{(r+1)!} \sum_{i=1}^{r+1}\mu_{i}, \frac{1}{r!} \mu_{1}\right).
\end{displaymath}
In particular, there exists $t$ such that the corresponding integral
gives $\frac{\nu}{(r+1)!}$.  Consider the semipositive toric metric
$(\|\cdot\|_{v})_{v}$ on $H$ given by, for $v=\infty$, the toric
metric associated to $\theta_{t}$ and, for $v\ne\infty$, the canonical
metric. A straightforward calculation shows that this metric satisfies
the required conditions.
\end{proof}

For the right hand inequality in Theorem \ref{thm:6},
we can relax the hypothesis of semipositivity of the metrized
$\R$-divisor, by replacing the height by the arithmetic
volume or the $\chi$-arithmetic volume of the divisor. In our present
toric setting, the obtained lower bound of the essential minimum in
terms of the $\chi$-arithmetic 
volume extends \cite[Lemme 5.1]{ChambertLoirThuillier:MMel} to
arbitrary global fields and metrized $\R$-divisors.

\begin{prop} \label{prop:5} Let $X$ be a proper toric variety over $\K$
  of dimension $n$ and $\ov D$ a toric metrized $\R$-divisor on $X$
  such that $D$ is big. Then
\begin{equation} \label{eq:3}
\upmu^{\ess}_{\ov D}(X)\ge     \frac{\avol_{\chi}(\ov D)}{(n+1) \vol(D)}.
  \end{equation}
If $\ov D$ is pseudo-effective, then 
\begin{equation} \label{eq:19}
\upmu^{\ess}_{\ov D}(X)\ge     \frac{\avol(\ov D)}{(n+1) \vol(D)}.
  \end{equation}
\end{prop}

\begin{proof}
  For short write $\Delta=\Delta_{D}$. We first prove \eqref{eq:3}. By
  Corollary \ref{cor:4}, $\upmu^{\ess}_{\ov D}(X)\ge \vartheta _{\ov
    D}(x)$ for all $x\in \Delta $. Therefore, using the formula for
  the $\chi$-arithmetic volume of a toric metrized $\R$-divisor in
  \cite[Theorem 1]{BurgosMoriwakiPhilipponSombra:aptv} and the
  classical formula for the volume of a toric variety with respect to
  a toric divisor, we have
  \begin{multline}\label{eq:21}
        \avol_{\chi}(\ov D)=(n+1)!\int_{\Delta }\vartheta _{\ov
      D}\dd\vol_{M}\le
    (n+1)!\int_{\Delta }\upmu^{\ess}_{\ov D}(X)\dd\vol_{M}\\=
    (n+1)\upmu^{\ess}_{\ov D}(X)n!\vol(\Delta 
    )=(n+1)\upmu^{\ess}_{\ov D}(X)\vol(D),
  \end{multline}
  Since $D$ is big, we have that $\vol(D)>0$, and the inequality
  follows. 

  By Corollary \ref{cor:7}\eqref{item:7}, if $\ov D$ is
  pseudo-effective then $\upmu^{\ess}_{\ov D}(X)\ge \max(0,\vartheta
  _{\ov D}(x))$ for all $x\in \Delta $. The inequality \eqref{eq:19}
  follows similarly because, by using
  \cite[Theorem~1]{BurgosMoriwakiPhilipponSombra:aptv}  and Corollary
  \ref{cor:7}\eqref{item:7},
  \begin{multline}\label{eq:54}
        \avol(\ov D)=(n+1)!\int_{\Delta }\max(0,\vartheta _{\ov
      D})\dd\vol_{M}\\ \le
    (n+1)!\int_{\Delta }\upmu^{\ess}_{\ov D}(X)\dd\vol_{M}=
    (n+1)\upmu^{\ess}_{\ov D}(X)\vol(D).
  \end{multline}
\end{proof}

Now we will characterize when equality occurs in the lower bounds in
Proposition~\ref{prop:5}. First we need a technical lemma.

\begin{lem}
  \label{lemm:1}
  Let $\Psi\colon N_{\R}\to \R$ be a conic function such that
  $\stab(\Psi) $ has nonempty interior, and $f\colon N_{\R}\to \R$ a continuous
  function such that $|f-\Psi|$ is bounded. Let $u_{0}\in N_{\R}$ and
  $\gamma\in \R$. The following conditions are equivalent:
  \begin{enumerate}
\item \label{item:6} $f^{\vee}(x)= \langle x,u_{0}\rangle + \gamma$
  for all $x\in \stab(\Psi )$;
  \item \label{item:1} $\conc(f)(u)=\conc(\Psi)(u-u_{0}) -\gamma$ for
    all $u\in N_{\R}$; 
\item \label{item:2}  $ f(u_{0})=-\gamma$ and $f(u)\le
  \conc(\Psi)(u-u_{0}) -\gamma$ for
  all $u\in N_{\R}$.
  \end{enumerate}
\end{lem}

\begin{proof}
  Set $\Delta=\stab(\Psi)$, which is a convex subset of
  $M_{\R}$ and agrees with $\stab(f)$ by the hypothesis $|f-\Psi |$
  bounded.  

  \eqref{item:6} $\Rightarrow$ \eqref{item:1}: the function $f$ is
  asymptotically conic in the sense of \cite[Definition
  A.3]{BurgosMoriwakiPhilipponSombra:aptv}. Hence
  $\stab(\conc(f))=\stab(f)=\Delta$ and $\conc(f)= f^{\vee\vee}$.
  Thus
  \begin{equation}
    \label{eq:51}
    \conc(f)(u)=(f^{\vee})^{\vee}=\inf_{x\in \Delta }\langle x,u-u_{0}
    \rangle -\gamma. 
  \end{equation}
  Analogously
  \begin{displaymath}
    \conc(\Psi )(u)=(\Psi ^{\vee})^{\vee}=\inf_{x\in \Delta }\langle x,u
    \rangle.     
  \end{displaymath}
  Thus, the right hand side of  \eqref{eq:51} agrees with
  $\conc(\Psi)(u-u_{0}) 
  -\gamma$. 

\eqref{item:1} $\Rightarrow$ \eqref{item:6}: This follows from the
fact that $f^{\vee}=\conc(f)^{\vee}$

\eqref{item:6} $\Rightarrow$ \eqref{item:2}: Since, for any function
with non empty stability set,
\begin{equation}
  \label{eq:52}
  f(u)\le \conc(f)(u)
\end{equation}
the bound for $f(u)$
follows from the implication \eqref{item:6} $\Rightarrow$
\eqref{item:1}. Thus we only have to show that $f(u_{0})=-\gamma $. 

By equations \eqref{eq:51} and \eqref{eq:52}, we have $f(u_{0})\le
-\gamma $. Thus assume that $f(u_{0})=-\gamma -\varepsilon $ for 
some $\varepsilon >0$.

Let $x_{0}$ be a point in the interior of $\Delta $ and choose a norm
$\|\cdot\|$ on $N_{\R}$. By \eqref{eq:51} and \eqref{eq:52}, and using
that $x_{0}$ belongs to the interior of $\Delta $, we deduce
that there exists $K>0$ such that for all $u\in N_\R$
\begin{align}
   f(u)-\langle x_{0},u-u_{0}\rangle  \le
  \inf_{x\in \Delta } \langle x-x_{0},u-u_{0}\rangle - \gamma\le -K\|u-u_{0}\|-\gamma \label{eq:53}.
\end{align}

By the continuity of $f$ there is $\eta >0$ such that, if
$\|u-u_{0}\|\le \eta$ then  $$f(u)-\langle x_{0},u-u_{0}\rangle\le
-\gamma -\varepsilon /2.$$ By the inequality  \eqref{eq:53}, if
$\|u-u_{0}\|\ge \eta$, then  $f(u)-\langle x_{0},u-u_{0}\rangle\le
-\gamma -\eta K$. Put $s=\min(\varepsilon /2,\eta K)>0$. Hence
\begin{displaymath}
  f(u)\le \langle x_{0},u-u_{0}\rangle - \gamma -s
\end{displaymath}
for all $u\in N_\R$. Thus
\begin{multline*}
  f^{\vee}(x_0)=\inf_{u\in N_{\R}}\langle x_0,u \rangle -f(u) \ge
  \inf_{u\in N_{\R}}\langle x_0,u \rangle-\langle x_{0},u-u_{0}\rangle +
  \gamma +s = \langle x_0,u_{0} \rangle+
  \gamma +s, 
\end{multline*}
contradicting \eqref{item:6}. Therefore $f(u_{0})=-\gamma $ finishing
the proof of \eqref{item:2}.

\eqref{item:2} $\Rightarrow$ \eqref{item:6}: let $x\in \Delta$. We have
that
\begin{displaymath}
  f^{\vee}(x)= \inf_{u\in N_{\R}}\langle x,u\rangle -f(u). 
\end{displaymath}
Hence, the inequality in
\eqref{item:2} implies that
\begin{math}
  f^{\vee}(x)\ge  \langle
  x,u_{0}\rangle + \gamma
\end{math}
and, on the other hand, 
\begin{math}
  f^{\vee}(x)\le \langle x,u_{0}\rangle - f(u_{0}) = \langle
  x,u_{0}\rangle + \gamma, 
\end{math}
which implies the statement. 
\end{proof}

Recall that $H_{\F}\subset \bigoplus _{w\in \mathfrak{M}_{\F}}N_{\R}$ is the
hyperplane defined in \eqref{eq:16}.

\begin{prop} \label{prop:6} Let $X$ be a proper toric variety over $\K$
  of dimension $n$ and $\ov D$ a toric metrized $\R$-divisor on $X$
  such that $D$ is big. 
  \begin{enumerate}
  \item \label{item:10} The equality 
\begin{equation*}
\upmu^{\ess}_{\ov D}(X)=     \frac{\avol_{\chi}(\ov D)}{(n+1) \vol(D)}
  \end{equation*}
  holds if and only if there exist real numbers
  $(\gamma_{v})_{v}\in \bigoplus_{v\in \mathfrak{M}_{\K}} \R$, and vectors
  $(u_{v})_{v}\in H_{\K}\subset \bigoplus_{v\in \mathfrak{M}_{\K}}
  N_{\R}$, indexed by the set of places of $\K$, such that
  \begin{enumerate}
  \item \label{item:17} $\psi _{\ov D,v}(u_{v})=-\gamma_{v}$, for all $v\in
    \mathfrak{M}_{\K}$ and
  \item \label{item:18} $\psi _{\ov D,v}(u)\le \conc(\Psi_{D})(u-u_{v})
    -\gamma_{v}$ for all 
    $v\in \mathfrak{M}_{\K}$. 
  \end{enumerate}
\item \label{item:11} If $\ov D$ is big, then the
  equality  
\begin{equation*}
\upmu^{\ess}_{\ov D}(X)=     \frac{\avol(\ov D)}{(n+1) \vol(D)}
  \end{equation*}
  holds if and only if there exist
  $(\gamma_{v})_{v}\in \bigoplus_{v\in \mathfrak{M}_{\K}} \R$ and
  $(u_{v})_{v}\in H_{\K}\subset \bigoplus_{v\in \mathfrak{M}_{\K}}
  N_{\R}$ such that
  \begin{enumerate}
  \item \label{item:19} $\sum_{v} n_{v}\gamma_{v}> 0$,
  \item $\psi _{\ov D,v}(u_{v})=-\gamma_{v}$, for all $v\in
    \mathfrak{M}_{\K}$ and
  \item $\psi _{\ov D,v}(u)\le \conc(\Psi_{D})(u-u_{v}) -\gamma_{v}$ for all
    $v\in \mathfrak{M}_{\K}$. 
  \end{enumerate}
  \end{enumerate}
\end{prop}

\begin{proof} For short we write $\Delta =\Delta _{D}$.
  We first prove \eqref{item:10}.  By \eqref{eq:21}, the equality for
  the essential minimum holds if and only if, for all $x\in \Delta$,
 \begin{equation}
   \label{eq:23}
 \vartheta_{\ov D}(x) =
\upmu^{\ess}_{\ov D}(X).
 \end{equation}
 Since $\vartheta_{\ov D} = \sum_{v}n_{v}\vartheta_{\ov D, v}$, 
 the functions $\vartheta_{\ov D, v}$ are concave and the weights
 $n_{v}$ are positive, it
 follows that all the functions $\vartheta _{\ov D,v}$ are affine and
 their linear parts add to zero. Hence, \eqref{eq:23} holds if and
 only if there exists a collection of real numbers
 $\{\gamma_{v}\}_{v}$, with $\gamma_{v}=0$ for all but a finite number
 of $v$ and $(u_{v})_{v}\in H_{\K}$ such that, for all $x\in \Delta$,
  \begin{equation}\label{eq:24}
    \vartheta_{\ov D,v}(x) = \langle u_{v},x\rangle +\gamma_{v}.
  \end{equation}
  By \cite[Proposition 4.16(1)]{BurgosMoriwakiPhilipponSombra:aptv},
  the functions $|\psi_{\ov D,v}-\Psi_D |$ are bounded. Therefore, Lemma
  \ref{lemm:1} implies that \eqref{eq:24} is
  equivalent to the conditions \eqref{item:17} and \eqref{item:18}, since $\mu^\ess_{\ov D}(X) = \vartheta_{\ov D} = \sum_vn_v\gamma_v$.

  The proof of \eqref{item:11} is similar, but using equation
  \eqref{eq:54} and observing that Corollary \ref{cor:7}\eqref{item:9}
  implies the extra condition \eqref{item:19}.
\end{proof}

Proposition \ref{prop:6} also gives a criterion for when the right
inequality in Theorem \ref{thm:6} is an equality.
\begin{cor}\label{cor:3}
  Let $X$ be a proper toric variety over
  $\K$ of dimension $n$ and $\ov D$ a  semipositive toric metrized
  $\R$-divisor on
  $X$ such that $D$ is big. Then the equality 
  \begin{equation*}
    \frac{\h_{\ov D}(X)}{ \deg_{D}(X)}= (n+1)
    \upmu^{\ess}_{\ov D}(X)
  \end{equation*}
  holds if and only if 
  there exist
  $(\gamma_{v})_{v}\in \bigoplus_{v\in \mathfrak{M}_{\K}} \R$ and 
  $(u_{v})_{v}\in H_{\K}\subset \bigoplus_{v\in \mathfrak{M}_{\K}}
  N_{\R}$ such that, for $v\in \mathfrak{M}_{\K}$, 
  \begin{equation*}
    \psi _{\ov D,v}(u)=\Psi_{D}(u-u_{v}) -\gamma_{v} \quad \text{
      for all    } u\in N_{\R}.
  \end{equation*}
\end{cor}
\begin{proof}
  Since $\ov D$ is assumed to be semipositive, we have that $\h_{\ov
    D}(X)=\avol_{\chi}(\ov D)$ and all the functions $\psi _{\ov D,v}$
  are concave. Thus the corollary follows from Proposition
  \ref{prop:6}\eqref{item:10} and Lemma \ref{lemm:1}.
\end{proof}

\begin{rem}\label{rem:2}
  Observe that, if a metrized $\R$-divisor
  $\ov
  D=(D,(\|\cdot\|_{v})_{v\in \mathfrak{M}_{\K}})$ the equivalent
  conditions of Corollary \ref{cor:3}, then its metric is very
  close to the canonical
  metric. For instance, if there is an element $t\in \T(\K)$ such that
  $\val_{v}(t)=u_{v}$ then
  \begin{displaymath}
    \|\cdot\|_{v}=\e^{-\gamma _{v}}t^{\ast}\|\cdot\|_{\can,v}.
  \end{displaymath}  
\end{rem}

\section{Examples}
\label{sec:exmpl}

The previous results allow us to compute the successive minima of several
examples. The difficulty of the computations increases with the number
of places where the metric differs from the canonical one. The
following subsections are ordered increasingly according to this level
of difficulty. 

\subsection{Canonical metric}
\label{sec:canonical-metric}
As a first example, we show that the essential minimum of a toric
variety with respect to a pseudo-effective toric $\R$-divisor equipped
with the canonical metric at all the places as in Example \ref{exm:6},
is zero.

\begin{prop}\label{prop:13}
Let $X$ be a proper toric variety over $\K$ of dimension $n$ and $\ov D$ a toric
metrized $\R$-divisor with the canonical metric. Then
\begin{displaymath}
  \upmu^{\ess}_{\ov D}(X)=
  \begin{cases}
    0&\text{ if }D\text{ is pseudo-effective,}\\
    -\infty&\text{ otherwise.}
  \end{cases}
\end{displaymath}
Moreover, if $D$ is nef, then
\begin{math}
  \upmu^{i}_{\ov D}(X)=0
\end{math}
for $i=1,\dots,n+1.$
\end{prop}
\begin{proof}
  Since the metric of $\ov D$ is the canonical one, we have that, for
  all $v\in \mathfrak{M}_{\K}$, the local roof function $\vartheta
  _{\ov D,v}$ is zero on $\Delta_{D}$, and so the global roof function
  $\vartheta_{\ov D}$ is also zero on $\Delta_{D}$. The result then
  follows from Corollary \ref{cor:4}, since $\Delta _{D}\not=\emptyset$
  if and only if $D$ is pseudo-effective.

  If $D$ is nef, the result about the successive minima follows
  similarly from Theorem~\ref{thm:2}\eqref{item:23}.
\end{proof}

\subsection{Weighted $L^{p}$-metrics on toric varieties.}
\label{sec:Twisted-FS}

In the next three subsections we consider the case when only one
metric (the Archimedean one in $\Q$) differs from the canonical
one. This will allow us to use Corollary \ref{cor:8}\eqref{item:20}.

We  introduce a general family of Archimedean metrics we
toric varieties. 
To this end, let $X$ be a proper toric variety of dimension $n$ over
$\Q$, with fan $\Sigma $, $D$ a
nef toric divisor on $X$ and $\Delta =\Delta _{D}$ the polytope
associated to $D$. The support function associated to $D$ is the 
support 
function of $\Delta $. It is  given, for $u\in N_{\R}$, by
\begin{equation}\label{eq:32}
  \Psi (u)=\min_{m\in \Delta \cap M} \langle m,u\rangle=
  \min_{m\in \cF(\Delta) ^{0}} \langle m,u\rangle.
\end{equation}

  Let $\bfalpha=(\alpha _{m})_{m\in M\cap \Delta }$ be a
 collection of non-negative real numbers such that, if $m$ is a vertex
 of $\Delta $, then $\alpha
 _{m}>0$. Let
$\Lambda >0$ be a real number. We consider the metric on $\cO(D)$ over
$X_{0}(\C)$ given, for $p\in X_{0}(\C)$, by
\begin{equation*}
  \|s_{D}(p)\|_{\Lambda ,\bfalpha}=\bigg(\sum_{m\in \Delta \cap M}\alpha
    _{m}|\chi^{m}(p)|^{\Lambda } \bigg)^{\frac{-1}{\Lambda }}.
\end{equation*}
The function associated to this metric, $\psi _{\Lambda
  ,\bfalpha}\colon N_{\R}\to \R$, is given by 
\begin{displaymath}
  \psi_{\Lambda ,\bfalpha}(u)=\frac{-1}{\Lambda
  }\log\bigg(\sum_{m\in\Delta\cap M}\alpha _{m}\e^{-\Lambda \langle m,u\rangle}\bigg).
\end{displaymath}

\begin{prop}\label{prop:15} The function $\psi_{\Lambda ,\bfalpha}$ is
  concave and $|\psi _{\Lambda ,\bfalpha}-\Psi |$ is bounded. Therefore,
  the metric $\|\cdot\|_{\Lambda ,\bfalpha}$ extends to a continuous
  semipositive metric on $\cO(D)$ over $X(\C)$.
\end{prop}
\begin{proof}
  Each function $\alpha _{m}\e^{-\Lambda \langle m,u\rangle}$ is
  log-convex. Since sums of log-convex functions are log-convex
  \cite[\S~3.5.2]{Boyd:Vandenberghe:co}, we deduce that
  $\psi_{\Lambda ,\bfalpha}$ is concave.

  For the second statement we first observe that
  \begin{multline*}
    \min_{m\in\cF(\Delta )^{0}} \alpha_{m}
    \max_{m\in\cF(\Delta )^{0}}\e^{-\Lambda \langle
      m,u\rangle} \le
    \sum_{m\in\Delta\cap M}\alpha _{m}\e^{-\Lambda \langle m,u\rangle}
    \\\le \# (\Delta\cap M)
    \max_{m\in\Delta\cap M } \alpha_{m}
    \max_{m\in\Delta\cap M}\e^{-\Lambda \langle
      m,u\rangle}. 
  \end{multline*}
  Using the equality \eqref{eq:32}, we deduce that $|\psi _{\Lambda
    ,\bfalpha}-\Psi |$ is bounded.

  The last statement follows then from
  \cite[Theorem~4.8.1]{BurgosPhilipponSombra:agtvmmh}.
\end{proof}
Let $\ov D$ be the metrized divisor given by $D$, the metric
$\|\cdot\|_{\Lambda ,\bfalpha }$ at the Archimedean place and the
canonical metric at the non-Archimedean places. Hence, the adelic
family of functions associated to $\ov D$ is given by $\psi _{\Lambda
  ,\bfalpha}$ at the Archimedean place and by $\Psi $ at the
non-Archimedean places.

\begin{exmpl}\label{exm:8}
  When $\Delta $ is the standard simplex,
  the toric variety is the projective space and the divisor is the
  hyperplane at infinity. When
   $\Lambda =2$ and $\alpha
  _{m}=1$ for all $m\in \Delta \cap M=\cF(\Delta )^{0}$ we recover the
  Fubini-Study metric. When $\Lambda =2$ and $\alpha 
  _{m}$ are arbitrary positive numbers, we recover the case of the
  weighted Fubini-Study metric as
in \cite[Example 6.5]{BurgosMoriwakiPhilipponSombra:aptv}. For general
$\Lambda $ we obtain
  weighted versions of the $ L^{p}$ metric. 
\end{exmpl}

Thus the metrics we are considering in this section are the natural
generalization to arbitrary 
proper toric varieties over $\Q$ of the weighted Fubini-Study metric
and weighted 
$ L^{p}$-metric. In fact, they are  the inverse images of
the weighted Fubini-Study and weighted $ L^{p}$-metrics on the
projective space by a suitable toric morphism.

 We first compute the absolute minima of the orbits of $X$.

\begin{prop}\label{prop:2}
  Let $\sigma \in \Sigma $ and $F_{\sigma }\subset \Delta $ the
  corresponding face. Then
  \begin{displaymath}
    \upmu^{\abs}(O(\sigma ))=\frac{1}{\Lambda }\log \Big(\sum_{m\in
      F_{\sigma }\cap M}\alpha 
     _{m}\Big).
  \end{displaymath}
\end{prop}
\begin{proof}
  Let $N(\sigma )=N/(\R\sigma\cap N) $ and $M(\sigma )\subset M$ be the dual
  lattice. Choose $m_{0}\in F_{\sigma }\cap M$. The divisor $D'=D+\div
  (\chi^{m_{0}})$ intersects properly the closure of the orbit
  $V(\sigma )=\ov{O(\sigma )}$. The polytope of $D'|_{V(\sigma )}$ is
  $F_{\sigma }-m_{0}\subset M(\sigma )_{\R}$. The  metric of $D$
  induces an  metric on $D'|_{V(\sigma )}$. By
  \cite[Corollary~4.3.18]{BurgosPhilipponSombra:agtvmmh} at every
  non-Archimedean place the induced metric is the canonical
  metric. Let $\pi _{\sigma }\colon N_{\R}\to N(\sigma )_{\R}$ be the
  projection. By
  \cite[Proposition~4.8.9]{BurgosPhilipponSombra:agtvmmh}, the
  function associated to the metric on the Archimedean place is 
  given, for $v\in N(\sigma )_{\R}$, by
  \begin{displaymath}
    \psi (v)=\sup_{u\in \pi _{\sigma }^{-1}(v)}\frac{-1}{\Lambda
    }\log\bigg(\sum_{m\in \Delta \cap M}\alpha 
    _{m}\e^{-\Lambda \langle m-m_{0},u\rangle}\bigg).  
  \end{displaymath}
  Fix $v\in M(\sigma )_{\R}$ and $u\in \pi _{\sigma }^{-1}(v)$.
  If $m\in F_{\sigma }$, then $\langle m-m_{0},u\rangle$ does not
  depend on the choice of $u$ and agrees with $\langle
  m-m_{0},v\rangle$, when we consider $m-m_{0}\in M(\sigma )$. If
  $m\not \in F_{\sigma }$, we choose $u_{0}$ in the relative interior
  of $\sigma $. Then
  \begin{displaymath}
    \lim_{\lambda \to \infty} \langle m-m_{0},u+\lambda u_{0}\rangle=\infty.
  \end{displaymath}
  Hence, we deduce
  \begin{displaymath}
    \psi (v)=\frac{-1}{\Lambda
    }\log\bigg(\sum_{m\in F_{\sigma }\cap M}\alpha 
    _{m}\e^{-\Lambda \langle m-m_{0},v\rangle}\bigg).  
  \end{displaymath}
  By Corollary \ref{cor:8},
  \begin{displaymath}
    \upmu^{\abs}(O(\sigma ))=-\psi (0)=\frac{1}{\Lambda
    }\log\bigg(\sum_{m\in F_{\sigma }\cap M}\alpha 
    _{m}\bigg).
  \end{displaymath}
 \end{proof}

 We can compute now the successive minima of $X$ with respect to $\ov D$.

 \begin{thm} \label{thm:9} Let notation be as above. Then, for $
   i=1,\dots, n+1$,
   \begin{displaymath}
     \upmu^{i}_{\ov D}(X)=\min_{\sigma \in \Sigma 
       ^{i-1}}\frac{1}{\Lambda }\log \Big(\sum_{m\in F_{\sigma }\cap M }\alpha
     _{m}\Big).
   \end{displaymath}    
   If furthermore $D$ is ample, then 
   \begin{displaymath}
     \upmu^{i}_{\ov D}(X)=\min_{F\in \cF(\Delta)
       ^{n-i+1}}\frac{1}{\Lambda }\log \Big(\sum_{m\in F \cap M}\alpha
     _{m}\Big).
   \end{displaymath}    
 \end{thm}
 \begin{proof}
   The first part follows directly from Proposition \ref{prop:2},
   Lemma \ref{lemm:3}\eqref{item:5} and the observation that, if
   $\sigma \subset \tau $, then $F_{\tau }\subset F_{\sigma }$. The
   second statement follows from the
   first and the fact that, when $D$ is ample, the correspondence
   between cones of $\Sigma $ and faces of $\Delta $ gives a
   bijection between cones of dimension $i-1$ and faces of dimension
   $n-i-1$.
 \end{proof}

The example below and those in \S~\ref{sec:transl-subt-can} and
\ref{sec:transl-subt-fub} share a common setting that we summarize here.

\begin{setting}\label{setting:1} Let $\K$ be a global field, 
  $\P^{r}$ the projective space of dimension $r$ over $\K$ and $H$
  the divisor corresponding to the hyperplane at infinity. Then
  $\P^{r}$ is a toric variety and $H$ is an ample toric divisor.

  Let $e_1,\dots,e_r$ be the standard basis of
  $(\R^{r})^{\vee}= \R^{r}$ 
and set also $e_0=0$.
  The polytope associated to $H$ is the
  standard simplex of $\R^{r}$:
  \begin{displaymath}
    \Delta^{r} =\conv(e_0,\dots,e_r).
  \end{displaymath}
  The toric divisor $H$ corresponds to the support function of this polytope
  $\Psi_{\Delta ^{r}} \colon \R^{r}\to \R$, that is
  \begin{displaymath}
    \Psi_{\Delta ^{r}} (u_{1},\dots,u_{r})=\min(0,u_{1},\dots,u_{r}).
  \end{displaymath}

  Let $N$ be a lattice of rank $n$ and $M$ the dual lattice.  Let
  $\iota\colon N\hookrightarrow \Z^{r}$ be an injective linear map. We
  set $m_{j}=\iota^{\vee} e_{j}\in M$ for the $j$-th coordinate of
  $\iota$, $j=1,\dots, r$, and also $m_{0}=\iota^{\vee} e_{0}=0$.  Let
   $p=(p_0:\dots:p_r)\in\P^r_{0}(\K)\simeq (\K^{\times})^{r}$ be a
  rational point and consider the monomial map
  $\varphi_{p,\iota}\colon \T\to \P^{r}$ given, for $t\in \T$, by
  \begin{equation*}
    \varphi_{p,\iota}(t)=(p_{0}\chi^{m_{0}}(t):\dots: p_{r}\chi^{m_{r}}(t)).
  \end{equation*}
  The image $\Im(\varphi_{p,\iota})$ is the translate of a subtorus of
  the open orbit
  $\P^{r}_{0}\simeq \G_{m}^{r}$ by the point $p$. We set $Y$
  for its closure in $\P^{r}$.  

  Let $\Sigma$ be the complete fan on $N_{\R}$ induced by $\iota$ and
  $\Sigma_{\Delta^{r}}$, and set $\Psi=\iota^{*}\Psi_{\Delta^{r}}$. We
  denote by $X$ and $D$ the proper toric variety over $\K$ and the
  toric Cartier divisor on $X$ associated to this data.  Set $
  \Delta=\conv(m_{0},\dots, m_{r})\subset M_{\R}$. We can verify that
  $\Sigma$ coincides with the normal fan of $\Delta$ and that $\Psi$
  is the support function of this polytope. In particular, $\Psi$ is
  strictly concave on $\Sigma$, the divisor $D$ is ample, and $\Delta
  _{D}=\Delta$.
  
  Therefore, the monomial map $\varphi_{p,\iota}$ extends to a toric
  morphism $X\rightarrow \P^{r}$ that we denote also by
  $\varphi_{p,\iota}$ as in
  \cite[(3.2.3)]{BurgosPhilipponSombra:agtvmmh}. Let $Y$ denote the image of
  $\varphi_{p,\iota}$. If $\iota(N)$ is a saturated sublattice
  of $\Z^{r}$, then $X$ is the normalization of $Y$. In general, the
  map $X\to Y$ is finite and its degree is given by the index of the
  $\Z$-module $\iota(N)$ in its saturation.  The definition of $\Psi$
  implies that $D=\varphi_{p,\iota}^{*}H$.
\end{setting}

 \begin{exmpl}\label{exm:7}
   We place ourselves in the Setting \ref{setting:1} with $\K=\Q$ and
   $p=(1:\dots:1)$ the distinguished point of the principal orbit of $\P^{r}$. Thus
   we consider the projective space $\P^{r}$ as a toric variety. We
   equip the
   divisor at infinity $H$ with the Fubini-Study metric at
   the Archimedean place and the canonical metric at the
   non-Archimedean places. We denote $\ov
   H^{\FS}$ the obtained metrized divisor. As in Example \ref{exm:8}
   this corresponds to the standard simplex, $\Lambda =2$
   and $\alpha _{m}=1$. Thus Theorem \ref{thm:9} implies
   \begin{displaymath}
     \upmu^{i}_{\ov H^{\FS}}(\P^{r})=\frac{1}{2}\log(r+2-i).
   \end{displaymath}
   We consider now the
   metrized divisor on $X$ given by
 \begin{displaymath}
   \ov D^{\FS}=\varphi_{p,\iota}^{*}\ov H^{\FS}.
 \end{displaymath}
 Then, since $D$ is ample and the map $X\to Y$ is finite, by Theorem
 \ref{thm:9} and Proposition \ref{prop:9}\eqref{item:3},
 \begin{displaymath}
   \upmu^{i}_{\ov D^{\FS}}(X)=     \upmu^{i}_{\ov H^{\FS}}(Y)=\min_{F\in \cF(\Delta) 
     ^{n-i+1}}\frac{1}{ 2}\log( \# \{j\mid m_{j}\in F\}).
 \end{displaymath}    
 Hence we recover the computation of the 
 successive minima of subtori with respect to the Fubini-Study metric
 in \cite{Sombra:msvtp}.

   As an illustration, we consider the quadric $Q\subset \P^{3}$
   defined as the image of  the 
   monomial map
   \begin{displaymath}
     \P^2 \longrightarrow \P^3, \quad (t_{0},t_{1},t_{2}) \longmapsto
     (t_{0} : t_{0}t_{1}: t_{0}t_{2}: t_{1}t_{2}).
   \end{displaymath}
   The polytope $\Delta$ is the unit square $[0,1]^{2}$. Considering the
   lattice points in its different faces, we deduce that 
   \begin{displaymath}
     \upmu^{1}_{\ov H^{\FS}}(Q)=\log(2), \quad \upmu^{2}_{\ov
       H^{\FS}}(Q)=\frac{1}{2}\log(2), \quad  \upmu^{3}_{\ov H^{\FS}}(Q)=0.
   \end{displaymath}
 \end{exmpl}

 \subsection{Weighted projective spaces}
 \label{sec:metr-from-polyt}
 Let $\Delta \subset M_{\R}$ be a lattice simplex of dimension $n$ and
 $(X,D)$ the associated polarized toric variety over $\Q$.  Let
 $u_{0},\dots,u_{n}$ be a set of vectors of $N_{\R}$, orthogonal to
 the faces of $\Delta $ and pointing inwards.  The variety $X$ is a
 weighted projective space if and only if the primitive vectors
 colinear to $u_{0},\dots,u_{n}$ generate the lattice $N$ while, for a
 general lattice simplex $\Delta $, the toric variety $X$ is a fake
 weighted projective space, see \cite{Buczynska:fwps}.

 In this section we are going to consider a family of Archimedean
 metrics on this kind of polarized toric varieties. To this end
 choose a system of affine functions on $M_{\R}$,  
 $\ell_i(x)=\langle u_i,x\rangle-\lambda_i$, $i=0,\dots,n$, such
 that $$\Delta=\{x\in M_{\R}\mid \ell_i(x)\geq0,i=0,\dots,n\}.$$ Let
 $c_{i}$, $i=0,\dots,n$ be a collection of positive real numbers such that
 $\sum_{i=0}^nc_iu_{i}=0$. 
 We consider the function on $\Delta$ given by
 \begin{equation*}
   \vartheta(x) := -\sum_{i=0}^{n}c_i\ell_i(x)\log(\ell_i(x)).
 \end{equation*}
 This function is concave
 \cite[Lemma~6.2.1(1)]{BurgosMoriwakiPhilipponSombra:aptv}.
 We endow $D$ with the canonical metric
 at all the finite places of $\Q$ and with the metric associated to
 $\vartheta$ under the correspondence of \cite[Theorem
 4.8.1(2)]{BurgosPhilipponSombra:agtvmmh} at the Archimedean one. 
 This is a particular case of the metrics associated
 to polytopes described in \cite[\S~6.2]{BurgosPhilipponSombra:agtvmmh}. 

   Let $s_{D}$ be the toric section of $\cO(D)$, $m_{0},\dots,m_{n}$
   the vertices of $\Delta $, ordered in such a way that
   $\ell_{i}(m_{i})>0$, and 
   $\Lambda:=-\big(\sum_{i=0}^nc_i\lambda_i\big)^{-1}$. Note that
   $\Lambda >0$ because
   $-\sum_{i=0}^nc_i\lambda_i=\sum_{i=0}^nc_i\ell_{i}(x)>0$.

   The next proposition shows that the metrics considered in this
   section are a particular case of the metrics considered in the
   previous section.

 \begin{prop} \label{prop:7}
   The Legendre-Fenchel dual of $\vartheta $ is the concave function
   \begin{displaymath}
     \psi (u)=\frac{-1}{\Lambda }\log\bigg(
     \sum_{i=0}^{n}\Lambda c_{i}\e^{-\Lambda \langle m_{i},u\rangle}\bigg).
   \end{displaymath}
   Therefore the metric at the Archimedan place is given,
   for $p\in X_{0}(\C)$, by
 \begin{displaymath}
   \|s_D(p)\|_{\infty}=
     \bigg(\sum_{i=0}^n\Lambda c _{i}|\chi^{m_{i}}(p)|
     ^{\Lambda}\bigg)^{\frac{-1}{\Lambda}}.
   \end{displaymath}
 \end{prop}
 \begin{proof}
   We consider first the case of the simplex standard $\Delta ^{n}$
   and the concave function
   \begin{displaymath}
     \vartheta_{0} (x)=\frac{-1}{\Lambda }\sum_{i=0}^{n}x_{i}\log(x_{i}/\Lambda c_{i}),
   \end{displaymath}
   where we write $x_{0}=1-\sum_{i=1}^{n}x_{i}$.
   Arguing as in \cite[Example 2.4.3]{BurgosPhilipponSombra:agtvmmh},
   one checks that
   \begin{displaymath}
     \psi _{0}(u):= \vartheta_{0} ^{\vee}(u)=\frac{-1}{\Lambda
     }\log\bigg(\sum_{i=0}^{n}\Lambda c_{i}\e^{-\Lambda u_{i}}\bigg), 
   \end{displaymath}
   where $u=(u_{1},\dots,u_{n})$ and $u_{0}=0$.

   We now consider the function
   \begin{math}
     \varphi\colon M_{\R}\to \R^{n}
   \end{math}
   given by
   \begin{displaymath}
     \varphi(x)=\Big(\frac{\ell_{1}(x)}{\ell_{1}(m_{1})},\dots,\frac{\ell_{n}(x)}{\ell_{n}(m_{n})}\Big).
   \end{displaymath}
   This affine function sends $\Delta $ to the standard simplex.
   Note that, by the definition of $c_{i}$ and $\Lambda $, we have
   $$\ell_{i}(m_{i})=\frac{1}{\Lambda c_{i}} \quad \text{ and }\quad 
   \sum_{i=0}^{n}\frac{\ell_{i}(x)}{\ell_{i}(m_{i})}=1.$$  
   Using these relations one can verify that
   \begin{math}
     \vartheta =\varphi^{\ast}\vartheta _{0}.
   \end{math}
   We write $\varphi(x)=H(x)+a$, where $H$ is a linear isomorphism
   and $a\in \R^{n}$. Then, by  \cite[Proposition
   2.3.8(2)]{BurgosPhilipponSombra:agtvmmh},
   \begin{equation}\label{eq:8}
     \psi (u)=(H^{\vee})_{\ast}(\psi
     _{0}-a)(u)=\psi_{0}((H^{\vee})^{-1}u)-\langle H^{-1}a,u \rangle.
   \end{equation}
   Let $e_{1},\dots , e_{n}$ be the standard basis of $\R^{n}$ and put $e_{0}=0$.
   Since $\varphi^{-1}$ sends $e_{i}$ to $m_{i}$, we deduce that
   \begin{displaymath}
     (H^{\vee})^{-1}u=(\langle m_{1}-m_{0},u\rangle,\dots,\langle m_{n}-m_{0},u\rangle)
   \end{displaymath}
   and that $H^{-1}a=m_{0}$. Substituting this in equation
   \eqref{eq:8} we obtain the first statement of the proposition. The
   second statement follows directly from the first. 
 \end{proof}

  The faces of $\Delta$ are in one-to-one correspondence with the nonempty
  subsets $ I\subset \{0,\dots,n\}$ by the formula
  \begin{displaymath}
        F_{I} =\{x\in\Delta\mid \ell_j(x)=0, j\notin I\}.
  \end{displaymath}
  Therefore, Proposition \ref{prop:7} and Theorem \ref{thm:9} imply that the successive minima of
  $X$ are given by
    \begin{displaymath}
    \upmu^{i}_{\ov D}(X)=\min_{\substack{I\subset\{0,\dots,n\} \\ \# I=n-i+2}}
    \left(\frac{1}{\Lambda}\log\Big(\sum_{j\in I}\Lambda c_{j}\Big)\right), 
    \quad i=1,\dots, n+1.
  \end{displaymath}

  In contrast with the previous example, for the metrics of this
  section we can also compute explicitly the height of $X$ with respect to
  $D$ \cite[(6.2.4)]{BurgosPhilipponSombra:agtvmmh}:
  \begin{displaymath}
    \frac{\h_{\ov D}(X)}{\deg_{D}(X)}=\frac{n+1}{\Lambda
    }\sum_{j=2}^{n+1}\frac{1}{j}+\frac{1}{\Lambda
    }\sum_{i=0}^{n}\log(\Lambda c_{i}).
  \end{displaymath}

\subsection{Toric bundles}
In this section, we compute the successive minima of the toric
bundles that we considered in
\cite[\S~7.2]{BurgosPhilipponSombra:agtvmmh}. Let $n\ge 0$ and write
$\P^n=\P^n_{\Q}$ for short.  Let $a_r\geq \dots\geq a_0\geq 1$ be
integers, consider the bundle $\P(E)\rightarrow\P^n$ of hyperplanes of
the vector bundle
$$
E=\cO(a_0)\oplus \cO(a_1)\oplus \dots\oplus
\cO(a_r) \longrightarrow \P^n,
$$
where $\cO(a_{j})$ denotes the $a_{j}$-th power of the universal line
bundle of $\P^{n}$. This bundle is a smooth toric variety over
$\Q$ of dimension $n+r$. 

We consider its universal
line bundle $\cO_{\P(E)}(1)$, that is the line bundle corresponding to
the Cartier divisor $D:=a_0 
D_0+D_1$, where $D_0$ denotes 
the inverse image in $\P(E)$ of the hyperplane at infinity of
$\P^n$ and $D_1=\P(0\oplus \cO(a_1)\oplus\dots\oplus
\cO(a_r))$. 
It is an ample Cartier divisor. 

As explained in \cite[\S~7.2]{BurgosPhilipponSombra:agtvmmh}, there is
a standard splitting $N_{\R}= \R^{n+r}$.
This splitting gives us coordinates
$(x,y)=(x_{1},\dots,x_{n},y_{1},\dots,y_{r})$ on $M_{\R}=
\R^{n+r}$. We set $y_0=1-\sum_{i=1}^ry_i$, $L(y)=\sum_{j=0}^ra_jy_j$
and $x_0=L(y)-\sum_{i=1}^nx_i$.  With this notation, the polytope
associated to $D$ is
\begin{equation*}
\Delta_{D} :=
\left\{(x,y)\in\R^{n+r}\mid x_0,\dots,x_{n},y_{0},\dots, y_r\geq 0\right\} .
\end{equation*}

At the Archimedean place, we equip $D$ with the smooth metric induced
by the Fubini-Study metric in each sumand of $E$. 
If we denote by $s_{D}$ the toric section associated to $D$, this
metric is given, for $(z,w)\in (\C^{\times})^{n+r}\simeq
\P(E)_{0}(\C)$, by 
\begin{equation}\label{eq:1}
  \|s_{D}(z,w)\|_{\infty}= \bigg(\sum_{j=0}^{r}|w_{j}|^{2}\bigg( \sum_{i=0}^{n}|z_{i}|^{2}\bigg)^{a_{j}}  \bigg)^{-\frac12} 
\end{equation}
with $w_{0}=z_{0}=1$. We also equip $D$ with the canonical metric at
the non-Archimedean places and we denote by $\ov D$ the obtained
semipositive metrized divisor. 
Clearly,
\begin{equation}
  \label{eq:11}
      \|s_{D}(z,w)\|_{\infty}= \bigg(\sum_{m\in \Delta_{D} \cap M}\alpha
    _{m} |\chi^{m}(z,w)|^{2}\bigg)^{-\frac{1}{2}}
\end{equation}
for certain weights $\alpha_{m}\in \R_{\ge0}$. 
Hence, this is again a particular case of the
metrics considered in \S~\ref{sec:Twisted-FS}.

\begin{prop}\label{prop:8} With the previous notation
\begin{displaymath}
\upmu^\ess_{\overline{D}}(\P(E))= 
\frac{1}{2}\log\left((n+1)^{a_0}+\dots+(n+1)^{a_r}\right).
\end{displaymath}
\end{prop}
\begin{proof}
By Theorem \ref{thm:9}, we deduce from \eqref{eq:11} that
\begin{displaymath}
  \upmu^\ess_{\overline{D}}(\P(E))= \frac{1}{2}\log \Big( \sum_{m\in \Delta_{D}}\alpha_{m}\Big).
\end{displaymath}
To compute the sum inside the logarithm, it is enough to evaluate the
expression for
$\|s_{D}(z,w)\|_{\infty}^{-2}$ given by \eqref{eq:1} at $w_{j}=z_{i}=1$ for all
$j,i$, which gives the stated formula.
\end{proof}

\begin{prop}
 Let $1\leq i\leq n+r+1$. Then
 \begin{displaymath}
   \upmu^i_{\ov D}(\P(E))=\min_{\max(0,i-r-1)\leq\ell\leq
  \min(i-1,n)}\bigg(\frac{1}{2}\log
  \Big(\sum_{j=0}^{r+1-i+\ell}(n+1-\ell)^{a_j}\Big)\bigg).
 \end{displaymath}
\end{prop}
\begin{proof}
  By Theorem \ref{thm:9}
     \begin{equation}\label{eq:47}
     \upmu^{i}_{\ov D}( \P(E))=\min_{F\in \cF(\Delta_{D})
       ^{n+r-i+1}}\frac{1}{2}\log \Big(\sum_{m\in F \cap M}\alpha
     _{m}\Big),
   \end{equation}    
   where the weights $\alpha _{m}$ in \eqref{eq:11} are defined by the
   equation
   \begin{equation}\label{eq:27}
     \sum_{j=0}^{r}|w_{j}|^{2}\bigg(
     \sum_{i=0}^{n}|z_{i}|^{2}\bigg)^{a_{j}}=
     \sum_{m\in \Delta_{D} \cap M}\alpha
    _{m} |\chi^{m}(z,w)|^{2}
   \end{equation}
   with $z_{0}=w_{0}=1$.
   In order to compute $\sum_{m\in F \cap M}\alpha_{m}$ easily without
   developing equation \eqref{eq:27} we use the following trick. Let
   \begin{displaymath}
     m=(x_{1},\dots,x_{n},y_{1},\dots,y_{r})\in \Delta _{D}\cap M,
   \end{displaymath}
   as before we put $y_{0}=1-\sum_{j=1}^{r}y_{j}$ and
   $x_{0}=L(y)-\sum_{i=1}^{n}x_{i}$ and  write
   \begin{displaymath}
     \chi_{0}^{m}(z_{0},\dots,z_{n},w_{0},\dots,w_{r})=\prod_{i=0}^{n}z_{i}^{x_{i}}
     \prod_{j=0}^{r}w_{j}^{y_{j}}.
   \end{displaymath}
   We claim that 
   \begin{equation}\label{eq:30}
     \sum_{j=0}^{r}|w_{j}|^{2}\bigg(
     \sum_{i=0}^{n}|z_{i}|^{2}\bigg)^{a_{j}}=
     \sum_{m\in \Delta_{D} \cap M}\alpha
    _{m} |\chi_{0}^{m}(z,w)|^{2}
   \end{equation}
   for all $(z_{0},\dots,z_{n},w_{0},\dots,w_{r})\in \C^{n+r+2}$.
   We consider the bigrading that gives $z_{i}$ bidegree $(1,0)$ and
   $w_{j}$ bidegree $(-a_{j},1)$. Then both sides of equation
   \eqref{eq:30} are bihomogeneous of bidegree $(0,1)$ and they agree
   whenever  $z_{0}=w_{0}=1$. Therefore they agree on $\C^{n+r+2}$. 

   The faces of $\Delta_{D}$ of dimension $n+r-h$ are the slices obtained
   cutting $\Delta_{D}$ by hyperplanes $x_{i}=0$, $i\in I$ and $y_{j}=0$,
   $j\in J$, with $I\subsetneq\{0,\dots,n\}$, $J\subsetneq\{0,\dots,r\}$
   and $\# I+\# J=h$. We denote $F_{I,J}$ such a face.  Consider
the point $p_{I,J}\in \C^{n+r+2}$ given by $z_{i}=0$ if $i\in I$,
$z_{i}=1$ if $i\not\in I$,   $w_{j}=0$ if $j\in J$,
$w_{j}=1$ if $j\not\in J$. This point satisfies
\begin{displaymath}
  \chi_{0}^{m}(p_{I,J})=
  \begin{cases}
    1&\text{ if }m\in F_{I,J}\cap M,\\
    0&\text{ if }m\in (\Delta _{D}\setminus F_{I,J})\cap M.
  \end{cases}
\end{displaymath}
Evaluating \eqref{eq:30} at
the point $p_{I,J}$ we obtain
\begin{displaymath}
  \sum_{j\not \in J}(n+1-\# I)^{a_{j}}=\sum_{m\in F_{I,J}\cap M}\alpha _{m}.
\end{displaymath}
Thus, by \eqref{eq:47},  $\upmu^i_{\ov D}(\P(E))$ is the
minimum of   
$\frac{1}{2}\log\left(\sum_{j\notin J}(n+1-\# I)^{a_{j}}\right)$
over all $I,J$ satisfying $\# I+\# J  = i-1$. We obtain the
result by writing
$\ell=\# I$ and taking into account that we
ordered the $a_j$ so that $a_r\geq \dots\geq a_0\geq 1$.
\end{proof}

In particular, when $i=1$ then $\ell$ necessarily takes the value $0$
and we recover Proposition \ref{prop:8}. Whereas for $n+1\leq i\leq
n+r+1$ it can be shown that the minimum is attained with $\ell=n$. For
this value the sum
inside the logarithm equals $n+r+2-i$ and we get
$$\upmu^i_{\ov D}(\P(E)) = \frac{1}{2}\log\left(n+r+2-i\right)\ \mbox{for}\
i=n+1,\dots,n+r+1.$$ 

Remarkably, as in \S~\ref{sec:metr-from-polyt}, in this example the
roof function and the height of $\P(E)$ with respect to $\ov D$ are 
also computed, see \cite[\S~7.2]{BurgosPhilipponSombra:agtvmmh}.

\begin{exmpl}
  The particular case $n=r=1$ corresponds to
  the Hirzebruch surfaces: for $b\ge 0$, we have
  $\F_{b}=\P(\cO(0)\oplus \cO(b))\simeq
  \P(\cO(a_{0})\oplus \cO(a_{0}+b)) $ for any $a_{0}\ge1$. Although
  the surface does not depend on the choice of $a_{0}$, the divisor
  does. We set $a_{1}=a_{0}+b$. Then we obtain
  \begin{displaymath}
    \upmu^{\ess}_{\ov
      D}(\F_{b})=\frac{1}{2}\log(2^{a_{0}}+2^{a_{1}}),\quad
    \upmu^{2}_{\ov
      D}(\F_{b})=\frac{1}{2}\log(2),\quad
        \upmu^{\abs}_{\ov
      D}(\F_{b})=0.
  \end{displaymath}
\end{exmpl}

\subsection{Translates of subtori with the canonical metric}
\label{sec:transl-subt-can}

In this section and the next one we study examples where more
that one $v$-adic metric may be different from the canonical one.
We place ourselves in Setting
\ref{setting:1}. We equip $H$ with
the canonical metric at all the places and denote $\ov H^{\can}$
the obtained toric metrized divisor. Write $\ov D=\varphi_{p,\iota}\ov
H^{\can}$. Note that the metric induced on $D$ is not necessarily the
canonical one.

\begin{prop} \label{prop:10}
With the previous notation, for each $v\in \mathfrak{M}_{\K}$ let
$\vartheta_{v}\colon\Delta\to \R$ be the function parametrizing the upper
envelope of the polytope 
\begin{equation*}
  \wt \Delta _{v}:=\conv((m_{0},\log|p_{0}|_{v}), \dots, (m_{r},\log|p_{r}|_{v}))
  \subset M_{\R}\times \R
\end{equation*}
  and set $\vartheta=\sum_{v\in
    \mathfrak{M}_{\K}}n_{v}\vartheta_{v}$. Then, $\vartheta $ is the
  roof function of $\ov D$. In particular, for $i=1,\dots,
  n+1$, 
  \begin{displaymath}
     \upmu^{i}_{\ov D}(X) = \upmu^{i}_{\ov D}(Y) = \min_{F\in
       \cF(\Delta)^{n-i+1}} \max_{x\in
    F}\vartheta(x).
  \end{displaymath}
\end{prop}
\begin{proof}
  By \cite[Example 5.1.16]{BurgosPhilipponSombra:agtvmmh}, the
  function $\vartheta$ coincides with the roof function of $\ov
  D$. Since $D$ is ample and $\ov D$ is semipositive, Theorem
  \ref{thm:4} then gives the formula for the successive 
  minima of $X$. The fact that the successive minima of $X$ and $Y$
  coincide follows from Proposition~\ref{prop:9}\eqref{item:3}.
\end{proof}
 
By Proposition \ref{prop:10}, the computation of the  successive minima
of a translate of a 
subtori and of its normalization amounts to the computation of the
maximum of a piecewice affine function over a polytope. This is a
problem of linear programming. To do this in a concrete case, consider
the polytopes $\wt \Delta _{v}$ and the functions $\vartheta _{v}$ in
Proposition \ref{prop:10} and apply the following steps:

\begin{enumerate}[label=\rm (\alph*)]
\item \label{item:12} for each $v$ such that $\vartheta_{v}\not\equiv
  0$, compute the 
  regular subdivision $\Pi_{v}$ of $\Delta$ given by the
  projection of the faces of the polytope $\wt \Delta_{v}$;  

\item \label{item:13} compute a subdivision $\Pi$ refining $\Pi_{v}$
  for all $v$. This subdivision can be constructed by intersecting
  all the polyhedra in the different $\Pi _{v}$ as in \cite[Definition
  2.1.8]{BurgosPhilipponSombra:agtvmmh}; 

\item \label{item:14} the function $\vartheta $ is affine on each
  polytope of $\Pi $.
  Hence, for each face $F$ of $\Delta$, the maximum $\max_{x\in
    F}\vartheta(x)$ is realized at a vertex of $\Pi $ and to
  compute it we only
  need the values of  $\vartheta$ at the finite set $F\cap \Pi ^{0}$.
  Thus we obtain
  \begin{displaymath}
    \upmu^{i}_{\ov D}(X) = \upmu^{i}_{\ov D}(Y) = \min_{F\in
       \cF(\Delta)^{n-i+1}} \max_{x \in F\cap \Pi^{0}} \vartheta(x).
  \end{displaymath}
\end{enumerate}

Observe that, for each place $v$, the vertices of the subdivision $\Pi
_{v}$ in \ref{item:12} are lattice 
points. If the dimension of $Y$ is one, this implies that we can
choose $\Pi $ in \ref{item:13} such that all its vertices are
lattice points. This is 
the case in the example in the introduction. By contrast, in
higher dimension, we may need $\Pi $ to have non-lattice vertices as
shown in the next example.

\begin{exmpl}\label{exm:3}
Consider the quadric $S\subset \P^{3}$ defined as the closure of the
monomial map
\begin{displaymath}
  \T^2 \longrightarrow \P^3, \quad (t_{1},t_{2}) \longmapsto
  (1 : 2 t_{1}: 4 t_{2}: t_{1}t_{2}).
\end{displaymath} 
As before let $\ov D$ be the restriction of the metrized divisor $\ov
H$ to $S$.
The corresponding $v$-adic roof functions are described by the diagram
in Figure 
\ref{fig:localrfq}. These functions are the minimal concave piecewise
affine functions on the square with the prescribed values at the
vertices. The subdivisions $\Pi _{v}$ are also given in the diagram. 

\captionsetup[subfigure]{labelformat=empty}

\begin{figure}[h]
  \centering
  \begin{subfigure}[1]{0.32\textwidth}
    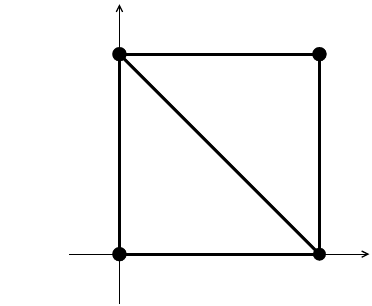
    \caption{$v=\infty$}
  \end{subfigure}
  \hspace*{-1mm}
  \begin{subfigure}[2]{0.32\textwidth}
    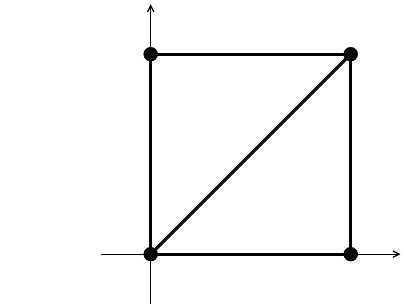    
    \caption{$v=2$}
  \end{subfigure}
  \hspace*{7mm}
  \begin{subfigure}[3]{0.26\textwidth}
    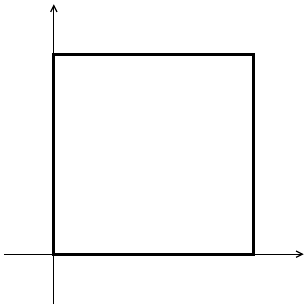    
    \caption{$v\ne\infty, 2$}
  \end{subfigure}
\caption{Local roof functions}\label{fig:localrfq}
\end{figure}

The global roof function and the subdivision $\Pi $ are given in
Figure \ref{fig:globalrfq}. From this picture, it follows that 
\begin{math}
  \upmu^{\ess}_{\ov D}(S)=\frac{3}{2} \log(2) \text{ and }
  \upmu^{2}_{\ov D}(S)=  \upmu^{3}_{\ov D}(S)=0.
\end{math}

\begin{figure}[h]
  \centering
    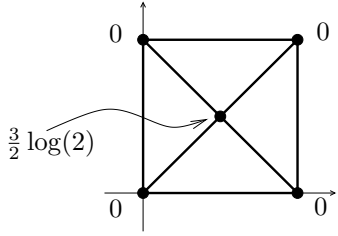
\caption{Global roof function}\label{fig:globalrfq}
\end{figure}
\end{exmpl}

\begin{rem}
  The method used in this example can be applied to compute the
  successive minima of any
  toric variety over $\K$ with a semipositive toric metrized $\R$-divisor $\ov D$
  such that $D$ is ample and the associated functions
  $\psi _{\ov D,v}$ are piecewise affine. In this case, for each $v$, the local roof
  function $\vartheta _{v}$ is not given by Proposition
  \ref{prop:10}, but it is computed as the Legendre dual of $\psi
  _{\ov D,v}$. Moreover, if one is only interested in the essential
  minimum, we can drop the ampleness and semipositiveness assumptions.   
\end{rem}

\subsection{Translates of subtori with the Fubini-Study metric}
\label{sec:transl-subt-fub}
We consider now the case when $X$ is a toric variety over $\Q$ and
$\ov D$ is a semipositive toric metrized $\R$-divisor, with $D$ ample
and such  that, for every non-Archimedean place 
$v\in \mathfrak{M}_{\K}$, the function $\psi _{\ov D, v}$ is
piecewise affine and for $v=\infty$, 
the function $\psi _{\ov D,
  v}$ is smooth. This is the case when the non-Archimedean
metrics are defined by means of a model and the Archimedean metric is
smooth, which is the situation classically considered in Arakelov
geometry.
 
Let $S\subset \mathfrak{M}_{\K}$ be the finite subset containing all
non-Archimedean places with $\psi _{\ov D,v}\not = \Psi _{D}$. 

\begin{lem} \label{lemm:10} With the previous notation, the essential
  minimum of
  $X$ with respect to $\ov D$ is computed  by applying the
  following steps.
  \begin{enumerate}[label=\rm (\alph*)]
  \item \label{item:16} For each place $v\in S$ we compute the
    function $\vartheta_{\ov D,v}$ as the Legendre-Fenchel dual of $\psi
    _{\ov D,v}$.
  \item Set $\vartheta _{S}=\sum _{v\in S}\vartheta _{\ov D,v}$ and
    compute its Legendre-Fenchel dual $\psi_{S}$.
  \item \label{item:15} Find a value $u_{0}\in N_{\R}$ such that
    \begin{displaymath}
      \partial \psi _{\ov D,\infty}(-u_{0})\in \partial \psi_{S}(u_{0}).
    \end{displaymath}
    In this condition, the left hand side is the
    differential of a smooth function and hence a vector, while the
    right hand side is the sup-differential of a concave piecewise
    affine function and hence is a set of vectors.
  \item The essential minimum of $X$ with respect to $\ov D$ is given by
    \begin{displaymath}
      \upmu^{\ess}_{\ov D}(X)=-\psi _{S}(u_{0})-\psi _{\ov D,\infty}(-u_{0}).
    \end{displaymath}
  \end{enumerate}  
\end{lem}
\begin{proof}
  By Corollary \ref{cor:8}, we know that
  \begin{displaymath}
    \upmu^{\ess}_{\ov D}(X)=
    -\Big(\big(\boxplus _{v\in S} \psi_{\ov D,v}\big)\boxplus
    \psi_{\ov D,\infty}  \Big)(0).
  \end{displaymath}
  By \cite[Proposition 2.3.1]{BurgosPhilipponSombra:agtvmmh}, the
  sup-convolution is dual to the sum and so
  $\boxplus _{v\in S} \psi_{\ov D,v}=\psi_{S}$. Hence
  \begin{displaymath}
    \upmu^{\ess}_{\ov D}(X)=
    -\Big(\psi_{S}\boxplus
    \psi_{\ov D,\infty} \Big)(0)=-
    \sup_{u\in N_{\R}}\big(\psi_{S}(u)+\psi_{\ov D,\infty}(-u)\big).
  \end{displaymath}
  Since the stability sets of $\psi_{S}$ and $\psi_{\ov D,\infty}$
  agree, by \cite[Theorem 16.4]{Rockafellar:ca}, the supremum is
  attained at some point. By the concavity of the functions, the supremum is
  attained at any point $u_{0}$ satisfying the condition
  \begin{displaymath}
    0\in \partial(\psi_{S}(u)+\psi_{\ov D,\infty}(-u))(u_{0}), 
  \end{displaymath}
  which is equivalent to the condition given in step \ref{item:15}.
\end{proof}

We place ourselves again in Setting \ref{setting:1} with $\K=\Q$
and we equip $H$ with the Fubini-Study metric at the Archimedean place
and the canonical metric at the non-Archimedean places. We denote $\ov
H^{\FS}$ the obtained toric metrized divisor and we set $\ov
D=\varphi_{p,\iota}^{\ast} \ov H^{\FS}$.

In this case the pair $(X,\ov D)$ satisfies the hypothesis of Lemma
\ref{lemm:10}. Moreover, for each Archimedean place $v$, the function
$\vartheta
_{\ov D,v}$ is given by the function $\vartheta _{v}$ in Proposition
\ref{prop:10}, hence step \ref{item:16} is already done. For the
Archimedean place the function $\psi _{\ov D,\infty}$ is given by
\begin{displaymath}
  \psi _{\ov
    D,\infty}(u)=-\frac{1}{2}\log\Big(\sum_{j=0}^{r}|p_{j}|^{2}\e^{-2\langle
  m_{j},u\rangle}\Big).
\end{displaymath}
We illustrate the recipe in Lemma \ref{lemm:10} in the following
examples, where $\ov D$  denotes the metrized divisor defined as before. 

\begin{exmpl}\label{exm:4}
  Let  $C \subset \P^2_{\Q}$ be the quadric curve over $\Q$ given as the image
of the  map  
\begin{displaymath}
  \P^1 \longrightarrow \P^2, \quad (t_{0}:t_{1}) \longmapsto
  \Big(t_{0}^{2} : \frac{1}{4}\, t_{0}t_{1}: \frac{1}{2}\, t_{1}^2\Big).
\end{displaymath} 
Then, for $v\not=2,\infty$, we
have $\psi _{\ov D,v}=\Psi _{D}$ and the corresponding metric agrees
with the canonical metric. Moreover
\begin{align*}
  \psi _{\ov D,2}(u)&=\min(0,u-2\log(2),2u-\log(2)),\\
  \psi_{\ov
    D,\infty}(u)&=
  -\frac{1}{2}\log\Big(1+\frac{1}{16}\e^{-2u}+\frac{1}{4}\e^{-4u}\Big).
\end{align*}
In this case $\psi _{S}=\psi _{\ov D,2}$ and $\partial \psi _{S}$ is
given by
\begin{displaymath}
  \partial \psi _{S}(u)=
  \begin{cases}
    2&\text{ if } u < -\log(2),\\
    [1,2]&\text{ if } u = -\log(2),\\
    1&\text{ if } -\log(2)< u < 2\log(2),\\
    [0,1]&\text{ if } u = 2\log(2),\\
    0&\text{ if } 2\log(2)< u.
  \end{cases}
\end{displaymath}
Then, analyzing the function $\partial \psi _{\ov D,\infty}(-u)$, we
deduce that the point $u_{0}$ that satisfies the condition in
step \ref{item:15} belongs to the interval $-\log(2)< u <
2\log(2)$. Thus we have to solve the equation $\partial \psi _{\ov
  D,\infty}(-u)=1$, whose only solution is
$u_{0}=\frac{1}{2}\log(2)$. Thus 
\begin{displaymath}
  \upmu^{\ess}_{\ov D}(C)=-\psi _{\ov D,2}\Big(\frac{1}{2}\log(2)\Big)-
  \psi_{\ov D,\infty}\Big(-\frac{1}{2}\log(2)\Big)=\frac{1}{2}\log(17).
\end{displaymath}
\end{exmpl}

\begin{exmpl}\label{exm:5}
  Consider the quadric surface of Example \ref{exm:3}, but recall that now
  $\ov D$ has the restriction of the Fubini-Study metric at the
  Archimedean place instead of the restriction of the canonical one.

  For $v\not = 2,\infty$, the function $\psi _{\ov D,v} =\Psi _{D}$.
  Hence $\psi _{S}=\psi_{\ov D,2}$. The function $\psi _{S}$ and its
  sup-differential is illustrated in Figure \ref{fig:psiS}. In this figure, we
  see a polyhedral decomposition of the plane. The two vertices of
  this polyhedral decomposition are the points $(-\log(2),\log(2))$
  and $(2\log(2),-2\log(2))$. 

  The function $\psi _{S}$ is affine in each of the
  four maximal polyhedra and its value on each polyhedra is given in
  the figure. In the interior of each of these polyhedra, the
  sup-differential contains a single vector also given in the figure. The
  sup-differential at a point belonging to a non-maximal polyhedra is the
  convex envelope of the sup-differentials of the neighbouring maximal
  polyhedra. For instance,
  \begin{alignat*}{2}
    \partial \psi _{S}(u_{1},u_{2})&=\conv((0,0),(1,0),(1,1)) \quad  &&\text{
      if } -\log(2)=
    u_{1}=-u_{2},\\
    \partial \psi _{S}(u_{1},u_{2})&=\conv((0,0),(1,1))  &&\text{ if }
    -\log(2)<u_{1}=-u_{2}<2\log(2).
  \end{alignat*}
\begin{figure}[h]
  \centering
    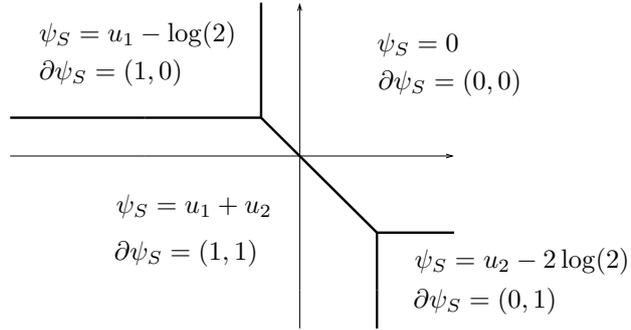
\caption{Function $\psi_{S}$ and its gradient}\label{fig:psiS}
\end{figure}
\end{exmpl}
The function $\psi _{\ov D,\infty}$ is given by
\begin{displaymath}
  \psi _{\ov D,\infty}(u_{1},u_{2})=
  -\frac{1}{2}\log(1+4\e^{-2u_{1}}+16\e^{-2u_{2}}+\e^{-2(u_{1}+u_{2})}).
\end{displaymath}
One checks that
\begin{displaymath}
  0< \frac{\partial \psi _{\ov D,\infty}}{\partial u_{1}}
  (u_{1},u_{2}),\frac{\partial \psi _{\ov D,\infty}}{\partial u_{2}}
  (u_{1},u_{2})< 1.
\end{displaymath}
This implies that a point $u_{0}$ satisfying the condition of step
\ref{item:15} belongs to the interval
$-\log(2)<u_{1}=-u_{2}<2\log(2)$. Thus we have to solve the equation
\begin{displaymath}
  \frac{\partial \psi _{\ov D,\infty}}{\partial u_{1}}
  (-u_{1},-u_{2})=\frac{\partial \psi _{\ov D,\infty}}{\partial u_{2}}
  (-u_{1},-u_{2}),\quad \text{ with }u_{1}=-u_{2}.
\end{displaymath}
This equation has a single solution at the point
$u_{0}=\frac{1}{2}(\log(2),-\log(2))$. Thus
\begin{displaymath}
  \upmu^{\ess}_{\ov D}(S)=-\psi
  _{S}(u_{0})-
  \psi_{\ov D,\infty}(-u_{0})=
  \log(3\sqrt{2}).
\end{displaymath}

\newcommand{\noopsort}[1]{} \newcommand{\printfirst}[2]{#1}
  \newcommand{\singleletter}[1]{#1} \newcommand{\switchargs}[2]{#2#1}
  \def\cprime{$'$}
\providecommand{\bysame}{\leavevmode\hbox to3em{\hrulefill}\thinspace}
\providecommand{\MR}{\relax\ifhmode\unskip\space\fi MR }
\providecommand{\MRhref}[2]{%
  \href{http://www.ams.org/mathscinet-getitem?mr=#1}{#2}
}
\providecommand{\href}[2]{#2}

\end{document}